\documentclass[hidelinks,11pt,reqno]{amsart}
\advance\textwidth2\evensidemargin
\usepackage{graphicx} 

\usepackage{todonotes}
\usepackage{amsfonts}
\usepackage{amsthm}
\usepackage{amsmath}
\usepackage{parskip}
\usepackage{mathrsfs}
\usepackage{amssymb}
\usepackage{xtab}
\usepackage{dsfont}
\usepackage{mathtools}
\usepackage{natbib}
\usepackage{csquotes}
\usepackage[hyphens]{url}
\usepackage{amssymb}
\usepackage{appendix}
\usepackage{tikz}
\usetikzlibrary{calc}
\usepackage{enumitem}
\usepackage{amscd}
\usepackage{setspace}
\usepackage{version}
\usepackage[dvipsnames]{xcolor}
\usepackage[raggedrightboxes]{ragged2e}

\usepackage[colorlinks=true,allcolors=black,pagebackref=false]{hyperref}
\usepackage{color}

\usepackage[noabbrev,capitalize,nameinlink]{cleveref}
\crefname{equation}{}{}
\crefname{conjecture}{Conjecture}{Conjectures} 
\AtBeginEnvironment{appendices}{\crefalias{section}{appendix}} 
\crefformat{enumi}{#2#1#3}
\crefrangeformat{enumi}{#3#1#4 to~#5#2#6}
\crefmultiformat{enumi}{#2#1#3}%
{ and~#2#1#3}{, #2#1#3}{ and~#2#1#3}

\renewcommand{\leq}{\leqslant}
\renewcommand{\geq}{\geqslant}
\newcommand{\Z}{\ensuremath{\mathbf{Z}}}
\newcommand{\N}{\ensuremath{\mathbf{N}}}
\newcommand{\R}{\ensuremath{\mathbf{R}}}

\newcommand{\E}{\ensuremath{\mathbb{E}}}

\newcommand\Supp{\operatorname{Supp}}
\newcommand{\prob}{\ensuremath{\mathbb{P}}}
\newcommand{\st}{\ensuremath{:}}
\newcommand{\ind}{\mathbf{1}}
\newcommand*\dif{\mathop{}\!\mathrm{d}}

\newcommand{\xb}{X}

\newcommand{\sth}{\textsuperscript{th }}
\newcommand{\xbd}{X_{\mathcal{D}}}
\newcommand{\xbee}{X_{\mathcal{E}}}

\newcommand{\cf}{\varphi_{Y}}

\newcommand{\Vertr}{\Vert_{\R/\Z}}
\newcommand{\nequiv}{\not\equiv}
\newcommand{\ru}[1]{e_{g-1}(#1)}
\newcommand{\Ru}[1]{e_{g-1}\Big(#1\Big)}
\newcommand{\fewnz}{\mathcal{B}}
\newcommand{\sx}[1]{s(#1)}

\newcommand{\V}[1]{\Vert #1 \Vertr}
\newcommand{\floor}[1]{\left\lfloor #1 \right\rfloor}
\newcommand{\sd}[1]{s_{g}(#1)}

\newtheorem{prop}{Proposition}[section]
\newtheorem{lemma}[prop]{Lemma}

\newtheorem{theorem}[prop]{Theorem}

\newtheorem{cor}[prop]{Corollary}
\newtheorem{claim}[prop]{Claim}

\newtheorem{defn}[prop]{Definition}

\newtheorem*{decoupling-rpt}{\cref{decoupling lemma}}
\newtheorem*{post-canc-rpt}{\cref{post cancellation}}
\newtheorem*{pre-canc-rpt}{\cref{pre block cancellation}}
\newtheorem*{llt-rpt}{\cref{Local limit theorem application version}}
\newtheorem*{lem32-rpt}{\cref{psi lemma}}
\newtheorem*{fm-rpt}{\cref{FM bound 2}}

\theoremstyle{remark}
\newtheorem*{remark}{Remark}

\newcommand{\md}[1]{\ensuremath{(\operatorname{mod}\, #1)}}

\newcommand{\nocontentsline}[3]{}
\let\origcontentsline\addcontentsline
\newcommand\stoptoc{\let\addcontentsline\nocontentsline}
\newcommand\resumetoc{\let\addcontentsline\origcontentsline}

\title{Niven numbers are an asymptotic basis of order 3}
\author{Kate Thomas}
\address{Mathematical Institute\\Woodstock Road\\ Oxford OX2 6GG, UK}
\email{katherine.thomas@maths.ox.ac.uk}

\begin{document}

\begin{abstract}
    A base-$g$ Niven number is a natural number divisible by the sum of its base-$g$ digits. We show that, for any $g\geq 3$, all sufficiently large natural numbers can be written as the sum of three base-$g$ Niven numbers. We also give an asymptotic formula for the number of representations of a sufficiently large integer as the sum of three integers with fixed, close to average, digit sums.
\end{abstract}

\maketitle
\tableofcontents
\section{Introduction}
A central question in additive number theory is to establish whether a given set of integers $\mathcal{S}$ is an \emph{asymptotic basis} for the integers, that is, to determine whether there exists a natural number $k$ such that any sufficiently large integer can be written as the sum of $k$ elements of $\mathcal{S}$. Here, $k$ denotes the \textit{order} of the basis.

Famously, Lagrange's theorem gives that the squares are a basis of order 4, and Waring's problem, solved by Hilbert, shows that $k$\sth powers are also an additive basis. Some interesting variants of Waring's problem consider $k$\sth powers of integers which have restrictions on their digits in some base. For example, Pfeiffer and Thuswaldner \cite{pfeiffer2007waring} show that the $k$\sth powers of integers with certain congruence conditions on their sums of digits in different bases is an asymptotic basis. More recently, Green \cite{green2025waring} established that, given any two digits which are coprime, the integers whose base-$g$ expansions consists of only these digits satisfy Waring's problem. Further references for additive bases coming from sets of integers with digit restrictions are given in the introduction of \cite{sanna2021additive}.

A \emph{base-$g$ Niven number} is a natural number that is divisible by its base-$g$ sum of digits. Such integers are also referred to as \textit{Harshad numbers}. It is shown in \cite{de2003counting}, and independently in \cite{mauduit2005distribution}, that the number of base-$g$ Niven numbers less than $x$ is asymptotically $\eta_{g}x/\log x$ for some constant $\eta_g>0$. 

It is conjectured that the set of base-$g$ Niven numbers is an asymptotic basis of order 2. In \cite{sanna2021additive} Sanna established, conditionally upon a certain generalisation of the Riemann Hypothesis, that the set of base-$g$ Niven numbers is an asymptotic basis with order growing linearly in $g$. Our result is the following unconditional statement.
\begin{theorem}\label{Main theorem 1}
   For any $g\geq 3$, the set of base-$g$ Niven numbers is an asymptotic basis of order 3.
\end{theorem}
We first count, via the circle method, the number of representations of a sufficiently large integer $M$ as the sum of three integers with a near-average digit sum. By showing that a proportion of such representations have each summand being a Niven number, we provide a lower bound for the number of representations of $M$ as the sum of three Niven numbers. To state these results more precisely, we need the following notation.

Let $\sd{n}$ denote the base-$g$ digit sum of $n$. Let $K\geq 1$ be a sufficiently large integer, and let $M\in (g^{K-1},g^{K}]$. For $k\in \N$, let
\[S_{g}(k)\coloneq \{n<g^{K}\st \sd{n}=k\} \textrm{ and }\mathcal{N}_{g}(k)\coloneq \{n\in S(k) \st k\mid n\}.\]
Thus $\mathcal{N}_{g}(k)$ is the set of Niven numbers in $S_{g}(k)$, that is, Niven numbers of a certain size with fixed digit sum. Let \begin{equation*}
    \mu_{K}\coloneq \frac{(g-1)K}{2},
\end{equation*}
which is the average digit sum of $n<g^{K}$. For a fixed choice of $k_{1},k_{2},k_{3}\in \N$, let $S_{i}\coloneq S_{g}(k_{i})$ for $i=1,2,3$, and let $\mathcal{N}_{i}\coloneq \mathcal{N}_{g}(k_{i})$ for $i=1,2,3$. Suppose that $k_{1},k_{2},k_{3}\in \N$ are such that
\begin{equation}\label{k cond S}
    |k_{i}-\mu_{K}|\leq C_{g} \textrm{ and } k_{1}+k_{2}+k_{3}\equiv M\md{g-1},
\end{equation}
where $C_g := g(g-1)\prod_{p\leq 10g^{2}}p$, and the product is over primes.  The specific choice of constant here is unimportant, and any sufficiently large value would do, but this value is large enough to ensure that $k_{i}$ with further desired properties exist. Let $r_{S_{1}+S_{2}+S_{3}}(M)$ be the number of representations of $M=s_{1}+s_{2}+s_{3}$, for $s_{i}\in S_{i}$. Our main result is the following theorem, from which we can later deduce the corresponding result for $\mathcal{N}_{i}\subset S_{i}$.
\begin{theorem}\label{Main theorem S count} Let $g,K$ and $M$ be integers such that $g\geq 3$, $K$ is sufficiently large in terms of $g$, and $M\in (g^{K-1},g^{K}]$. Suppose that $k_{1},k_{2},k_{3}$ satisfy \cref{k cond S}. Then
    \[r_{S_{1}+S_{2}+S_{3}}(M)=\frac{(g-1)M^{2}}{2(2\pi\sigma^{2}K)^{3/2}}(1+O_{g}((\log K)^{4}K^{-1/4})),\]
where $\sigma^{2}=(g^{2}-1)/12$.
\end{theorem}
Let $r_{\mathcal{N}_{1}+\mathcal{N}_{2}+\mathcal{N}_{3}}(M)$ be the number of representations of $M=n_{1}+n_{2}+n_{3}$, where $n_{i}\in \mathcal{N}_{i}$ for $i=1,2,3$. In order to relate the quantity $r_{S_{1}+S_{2}+S_{3}}(M)$ to $r_{\mathcal{N}_{1}+\mathcal{N}_{2}+\mathcal{N}_{3}}(M)$, we require some further conditions on the choice of $k_{1},k_{2},k_{3}$. Throughout, $(a,b)$ denotes the greatest common divisor of $a$ and $b$. Suppose that $k_{1},k_{2},k_{3}\in\N$ also satisfy, in addition to \cref{k cond S},
\begin{align}\label{k cond N}
  (k_{i},k_{j})=1 \textrm{ and }(k_{i},g)=1 \textrm{ for }i,j=1,2,3 \textrm{ and } i\neq j.
 \end{align}
Let 
\[c_{g}(k_{1},k_{2},k_{3})\coloneq \frac{4}{(g-1)^{2}}\prod_{i=1}^{3}(g-1,k_{i}).\]
Then we have the following theorem, from which \cref{Main theorem 1} is a corollary after showing that such a choice of $k_{1},k_{2}$ and $k_{3}$ exist. 
\begin{theorem}\label{Main theorem N count}
Let $g,K$ and $M$ be integers such that $g\geq 3$, $K$ is sufficiently large in terms of $g$, and $M\in (g^{K-1},g^{K}]$. Suppose that $k_{1},k_{2},k_{3}$ fulfil \cref{k cond S} and \cref{k cond N}, then
\[r_{\mathcal{N}_{1}+\mathcal{N}_{2}+\mathcal{N}_{3}}(M)=\frac{(g-1)^{2}c_{g}(k_{1},k_{2},k_{3})}{4k_{1}k_{2}k_{3}}r_{S_{1}+S_{2}+S_{3}}(M)+O_{g}(M^{2}K^{-29/6}).\]
In particular,
 \[r_{\mathcal{N}_{1}+\mathcal{N}_{2}+\mathcal{N}_{3}}(M)=c_{g}(k_{1},k_{2},k_{3})\frac{M^{2}}{(2\pi\sigma^{2})^{3/2}K^{9/2}}+O_{g}(M^{2}(\log K)^{4}K^{-19/4}),\]
 where $\sigma^{2}=(g^{2}-1)/12$.
\end{theorem}
Our methods can be adapted to show the analogous versions of \cref{Main theorem 1}, \cref{Main theorem S count} and \cref{Main theorem N count} for base 2, however there is one main technical difference. For readability, we do not give the details here, but comment on the necessary changes in \cref{Section: minor arc}.
\subsection{Notation}\label{notation}
Throughout, we consider $g$ to be fixed. Let $\Vert x \Vertr$ denote the distance of a real number $x$ to the nearest integer. We write $e(x)$ for $\exp(2\pi ix)$, and $\ru{x}$ for $e(x/(g-1))$. We require the quantity $\ell$, defined to be
\begin{equation}\label{l choice}
    \ell\coloneqq \lceil 384g^{3}\log K\rceil.
\end{equation}
Throughout, $\sigma=\sqrt{(g^{2}-1)/12}.$
\stoptoc
\section*{Acknowledgments}
\resumetoc
The author would like to thank Ben Green and Thomas Bloom for their invaluable advice and encouragement.
\section{An overview of the proof}
This section outlines the key steps to showing \cref{Main theorem 1}. The main work is to establish \cref{Main theorem S count}, and then to show that a sufficient number of these representations are of integers which are actually Niven numbers.

The primary tool for showing \cref{Main theorem S count} is the circle method. Notably in our application there are few major arcs; to show the base-$g$ result, we take $g-1$ major arcs, these are short intervals around the rationals $j/(g-1)$ for $j=0,\ldots, g-2$. The need to consider rationals of this form stems from the relation $\sd{n}\equiv n\md{g-1}$, which is the only congruence obstruction to \cref{Main theorem S count}. The contribution from these major arcs is handled in \cref{Section: major arc}, giving the main term in \cref{Main theorem S count}. We establish a uniform pointwise bound for the minor arcs; existing results cover a subset of the minor arcs, but are not sufficiently strong at certain minor arc points to give \cref{Main theorem S count}. This bound is proved in \cref{Section: minor arc}, and relies on a local limit theorem, the necessary consequences of which are given in \cref{section Psi}. 

In \cref{section divisibility}, we show \cref{Main theorem N count} by linking the number of representations of an integer as the sum of three integers with near-average digit sum, to representations where each summand is also a Niven number. 

Finally a choice of $k_{1},k_{2}$ and $k_{3}$ which will allow us to deduce \cref{Main theorem 1} from \cref{Main theorem N count} is given in \cref{section k choice}. 

\subsection{Counting the number of representations of an integer in $S_{1}+S_{2}+S_{3}$}
The main congruence obstruction to finding solutions to $M=s_{1}+s_{2}+s_{3}$ comes from the following fact. For all $n\in \Z$ and any base $g$, $g\geq 2$,
\begin{equation}\label{cong fact}
    \sd{n}\equiv n\md{g-1}
\end{equation}
Thus in order for $M=s_{1}+s_{2}+s_{3}$ to have solutions with $s_{i}\in S_{i}$, we must have
\begin{equation}\label{k cond: congruences}
    k_{1}+k_{2}+k_{3}\equiv M \md{g-1}.
\end{equation}
Let $\mu_{K}\coloneq (g-1)K/2$, then we also require that
\begin{equation}\label{k cond close to av}
    |k_{i}-\mu_{K}|\leq C_{g}.
\end{equation}
Note that $\mu_{K}$ is the average base-$g$ digit sum for $n<g^{K}$. By restricting to target digit sums $k_{i}$ that are close to the average value, we ensure that the sets $|S_{i}|$ are large. It is shown in \citep[Lemma~3]{mauduit2005distribution} that $|S_{g}(k)|$ is unimodal in $k$, with maximum size when $k=\floor{\mu_{K}}$. We require an asymptotic bound for the sizes of the sets $S_{i}$. The results of Mauduit and S\'{a}rk\"{o}zy \citep[Theorem~1]{mauduit1997arithmetic}, and Fouvry and Mauduit \citep[Theorem~1.1]{fouvry2005entiers} bound the size of $|S_{g}(k)|$, including for more general ranges of $k$ relative to $\mu_{K}$, however the error terms stated for these results are too large for our purposes; we need something that is $o_{g}(g^{K}K^{-1})$. As we only require a bound for $|S_{i}|$ when $k_{i}$ satisfies \cref{k cond close to av}, we are able to use the local limit theorem stated in \cref{section Psi} to get the following bound in this range.
\begin{cor}\label{S set bound}
For $k_{i}$ satisfying \cref{k cond close to av},
    \[|S_{i}| =\frac{g^{K}}{\sqrt{2\pi \sigma^{2}K}}+O_{g}(g^{K}K^{-3/2}),\]
    where $\sigma^{2}=(g^{2}-1)/12.$
   \end{cor}
Now we outline the proof of \cref{Main theorem S count}. Let $f_{i}(\theta)$ denote the Fourier transform $\widehat{\ind_{S_{i}}}(\theta)$, for $i=1,2,3$, so
\[f_{i}(\theta)\coloneq \sum_{n\in S_{i}}e(n\theta).\]
Then by orthogonality,
\begin{equation}\label{main S count}
    r_{S_{1}+S_{2}+S_{3}}(M)=\int_{\R/\Z}f_{1}(\theta)f_{2}(\theta)f_{3}(\theta)e(-M\theta)\dif\theta.
\end{equation}
We take the following simple major arcs:
\begin{equation*}
    \mathfrak{M}\coloneq \bigcup_{j=0}^{g-2}\Big [\frac{j}{g-1}-\varepsilon, \frac{j}{g-1}+\varepsilon\Big ],
\end{equation*}
where $\varepsilon\coloneq K^{3/4}g^{-K}/(g-1)$ throughout. Let the minor arcs be the remaining points, $\mathfrak{m}\coloneq (\R/\Z)\setminus \mathfrak{M}$. Intuition for these major arcs can be provided by the results of \cite{de2003counting,mauduit2005distribution,mauduit1997arithmetic}. These works show that the sets $S_{i}$ are well-distributed across the possible residue classes for a given modulus $m$, with the quality of relevant error terms depending on the size of $m$, and whether it is coprime to $g$ and $g-1$. The restriction to certain residue classes comes from \cref{cong fact}, but this is the only congruence restriction. As such, we expect cancellation in $\sum_{n\in S_{i}}e(n\theta)$ unless $\theta$ is very close to a multiple of $1/(g-1)$. At these points, there cannot be cancellation due to the following relationship, which holds for any $\theta \in \R$ and $x\in\Z$, 
\begin{equation}\label{trans inv}
    f_{i}(\theta+x/(g-1))=\ru{k_{i}x}f_{i}(\theta).
\end{equation}
As such, on intervals around multiples of $1/(g-1)$, $|f_{i}(\theta)|$ behaves identically to around $\theta=0$, hence these intervals are included in our major arcs. The relation \cref{trans inv} follows immediately from \cref{cong fact}, as
\[f_{i}\Big(\theta+\frac{x}{g-1}\Big)=\sum_{n\in S_{i}}e\Big (n\Big(\theta+\frac{x}{g-1}\Big) \Big )=\sum_{n\in S_{i}}e(n\theta)\ru{s_{g}(n)x}=\ru{xk_{i}}f_{i}(\theta).\]
 In \cref{Section: major arc} we evaluate the contribution to \cref{main S count} from the major arcs, as stated in the following proposition.
\begin{prop}\label{major arc contrib}
Let $K,M\geq 1$ be integers such that $K$ is sufficiently large and $M\in (g^{K-1},g^{K}]$. Then
    \[\int_{\mathfrak{M}}\prod_{i=1}^{3}f_{i}(\theta)e(-M\theta)=\frac{(g-1)M^{2}}{2(2\pi\sigma^{2}K)^{3/2}}+O_{g}(M^{2}(\log K)^{4}K^{-7/4}),\]
where $\sigma^{2}=(g^{2}-1)/12$.
\end{prop}
We show in \cref{Section: minor arc} that $f_{i}(\theta)$ is sufficiently small at $\theta\in \mathfrak{m}$ so that the contribution to \cref{main S count} from the minor arcs is subsumed into the error term of \cref{Main theorem S count}. Here is a precise statement.
\begin{prop}\label{minor arc bound} For all $\theta \in \mathfrak{m}$,
   \[f_{i}(\theta) \ll_{g} g^{K}K^{-5/4}.\] 
\end{prop}
Via Parseval's identity and the bound for $|S_{i}|$ given by \cref{S set bound}, this is sufficient to prove that the minor arcs contribute only to the error term in \cref{Main theorem S count}. In this work, we use a bound on $|f_{i}(\theta)|$ due to Fouvry and Mauduit, \cite{fouvry2005entiers}, however we remark that one can find other bounds on exponential sums over sets of integers with fixed digit sums in \cite{de2003counting, mauduit2017digits, shparlinski2024weyl}.

We apply the aforementioned result of Fouvry and Mauduit \cite{fouvry2005entiers} to bound $f_{i}(\theta)$ whenever $\theta$ is such that $(g-1)\theta$ has many non-zero digits in its \textit{centred} base-$g$ expansion. We define the notion of the centred base-$g$ expansion in \cref{section centred exp}; this expansion shifts the range of digits to be centred around zero. This leaves the task of bounding $f_{i}(\theta)$ for $\theta$ bounded away from translates of $1/(g-1)$, and with the specific form  
 \[\theta=\frac{1}{g-1}\Big (\frac{\varepsilon_{m_{1}}}{g^{m_{1}}}+ \ldots + \frac{\varepsilon_{m_{t}}}{g^{m_{t}}}+\eta\Big ),\]
where $m_{1}<\ldots <m_{t}\leq K$, $t\leq \ell$, $\varepsilon_{m_{i}}\in (-\tfrac{g}{2},\tfrac{g}{2}]\cap \Z$ and $|\eta|<g^{K}$. We use the fact that $t\leq \ell$ to approximate the value of $e(n\theta)$ for $n\in S_{i}$. It turns out that very little information about $n$ is actually needed for this task. Obviously the digits of $n$ must sum to $k_{i}$, as $n\in S_i$. We show that besides this, we also need to know the value of a very small number of digits of $n$, and the residue class modulo $g-1$ of the sum of a fixed subset of digits. 

By modelling the digits of $n$ as independent copies of a uniform random variable, we replace the condition that $n\in S_{i}$ by the probability that the random variables modelling the digits sum to $k_{i}$. We can estimate this term using the local limit theorem. This application of the local limit theorem uses that the target digit sums, $k_{i}$, are within a constant of the actual average digit sum $\mu_{K}=(g-1)K/2$, so that the bounds produced for $f_{1}(\theta),f_{2}(\theta)$ and $f_{3}(\theta)$ are identical.

Given \cref{major arc contrib} and \cref{minor arc bound}, we can deduce \cref{Main theorem S count}.
\begin{proof}[Proof of \cref{Main theorem S count}]
From \cref{main S count},
\begin{equation}
    r_{S_{1}+S_{2}+S_{3}}(M)=\int_{\R/\Z}\prod_{i=1}^{3}f_{i}(\theta)e(-M\theta)\dif\theta.
\end{equation}
\cref{major arc contrib} gives that the contribution to the above integral from the major arcs provides the main term in \cref{Main theorem S count}. We show that the contribution from the minor arcs $\mathfrak{m}$ is $O_{g}(g^{2K}K^{-7/4})$, and thus these points only contribute to the error term in \cref{Main theorem S count}. We have
\[\Big |\int_{\mathfrak{m}}\prod_{i=1}^{3}f_{i}(\theta)e(-M\theta)\Big |\leq \sup_{\theta\in \mathfrak{m}}|f_{1}(\theta)|\int_{\R/\Z}|f_{2}(\theta)||f_{3}(\theta)|\dif \theta.\]
By Cauchy-Schwarz and Parseval, this is
\[\leq |S_{2}|^{1/2}|S_{3}|^{1/2}\sup_{\theta\in \mathfrak{m}}|f_{1}(\theta)|.\]
\cref{S set bound} gives $|S_{i}|\ll_{g} g^{K}K^{-1/2}$ for $i=1,2$, and combining this with \cref{minor arc bound} gives
\[\Big |\int_{\mathfrak{m}}\prod_{i=1}^{3}f_{i}(\theta)e(-M\theta)\Big |\ll_{g} g^{2K}K^{-7/4}.\qedhere\]  
\end{proof}

\subsection{Restricting to Niven numbers with fixed digit sums}

To deduce \cref{Main theorem N count} from \cref{Main theorem S count}, we show that for \enquote{good} choices of $k_{1},k_{2},k_{3}$, the number of representations of $M$ as $s_{1}+s_{2}+s_{3}$ with $k_{i}\mid s_{i}$ for each $i$ is roughly a $(k_{1}k_{2}k_{3})^{-1}$ proportion of the total number of representations. Note that the conditions for the $k_{i}$ stated in the introduction give such a choice. 

It is shown by De Koninck, Doyon and K\'{a}tai in \cite{de2003counting} and independently, Mauduit, Pomerance and S\'{a}rk\"{o}zy in \citep[Theorem~C]{mauduit2005distribution}, that
 \begin{equation}\label{size N}
 \mathcal{N}_{i}\sim \frac{(k_{i},g-1)}{k_{i}}S_{i}.\end{equation}
The $(k_{i},g-1)$ term arises from the congruence relation between $n$ and $s_{g}(n)$ given in \cref{cong fact}. We show that indeed the expected proportion of representations are of sums of three Niven numbers, by showing that for appropriate $k_{1}
,k_{2}$ and $k_{3}$,
\[r_{\mathcal{N}_{1}+\mathcal{N}_{2}+\mathcal{N}_{3}}(M)=r_{S_{1}+S_{2}+S_{3}}(M)\prod_{i=1}^{3}\frac{(k_{i},g-1)}{k_{i}}(1+o_{g}(1)).\]
To show this, we exploit the fact that \cref{Main theorem S count} counts the number of ways to write $M$ as the sum of three integers with \textit{fixed} near-average digit sums. Let $g_{i}(\theta)$ be the Fourier transform of $\mathcal{N}_{i}$,
\[g_{i}(\theta)\coloneq \sum_{n\in \mathcal{N}_{i}}e(n\theta).\]
By considering fixed target digit sums $k_{i}$, we gain in that we can relate $g_{i}(\theta)$ to $f_{i}(\theta)$, detecting the condition that $k_{i}\mid n$ by orthogonality. This reduces the problem to one of understanding the Fourier transform $f_{i}(\theta)$ along translates of frequencies by multiples of $1/k_{i}$. 

 We show in \cref{Section: minor arc} that a strong bound for $f_{i}(\theta)$ is available whenever $(g-1)\theta$ has many digits in its centred base-$g$ expansion. As such, a \enquote{good} choice of $k_{1},k_{2},k_{3}$ requires that their reciprocals, and certain multiples thereof, have many non-zero digits in base $g$. This ensures that $g_{1}(\theta)$, $g_{2}(\theta)$ and $g_{3}(\theta)$ are only simultaneously large for $\theta \in \mathfrak{M}$, at which point we can use the results of \cref{Section: minor arc} to conclude the proof of \cref{Main theorem N count}.
\subsection{Probabilistic model for digits}\label{section prob model} 
Throughout, we switch to a probabilistic model, viewing the digits of $n$ as random variables to model the condition that $\sd{n}=k_{i}$ by a local limit theorem. To be precise, let us state some notation. Let $Y$ be a random variable uniformly taking values in $\{0,\ldots, g-1\}-(g-1)/2$. Throughout, $\sigma^{2}$ will denote the variance of $Y$, $\sigma^{2}=(g^{2}-1)/12$. The translation by $-(g-1)/2$ is to ensure that $Y$ is mean-zero; we will account for this shift where appropriate, and thus note that when $g$ is even, $Y$ is not integer valued. Let $X_{0},\ldots, X_{K-1}$ be independent and identically distributed copies of $Y$, and let
\[\xb\coloneq \sum_{j=0}^{K-1}X_{j}g^{j}.\]
Then $\xb+(g^{K}-1)/2$ is a uniform random integer supported on $\{0,\ldots, g^{K}-1\}$, with $j$\sth digit $X_{j}+(g-1)/2$. This follows by the uniqueness of base-$g$ expansions, as 
\[\xb+\frac{g^{K}-1}{2}=\sum_{j=0}^{K-1}\Big (X_{j}+\frac{g-1}{2}\Big )g^{j} \textrm{ and }X_{j}+\frac{g-1}{2}\in \{0,\ldots, \frac{g-1}{2}\}.\]
Moreover, this means that $\sd{\xb+(g^{K}-1)/2}=\sum_{j=0}^{K-1}X_{j}+\mu_{K}.$ For convenience, we define the following mean-zero digit sum function for $\xb$. For $\xb=\sum_{j=0}^{K-1}X_{j}g^{j}$, let $\sx{\xb}\coloneq \sum_{j=0}^{K-1}X_{j}$. Equivalently, $\sx{\xb}=\sd{\xb+(g^{K}-1)/2}-\mu_{K}$. Let
\begin{equation}\label{xi defn}
    \xi_{i}\coloneq k_{i}-\mu_{K}.
\end{equation}
Now we may replace the sum over $n\in S_{i}$ by an average over $\xb$,  
\begin{equation}\label{lemma expectation version f}
    f_{i}(\theta)=g^{K}e\Big (\frac{g^{K}-1}{2}\theta\Big )\E_{\xb}e(\xb\theta)\ind_{\sx{\xb}=\xi_{i}}.
\end{equation}
Recasting $f_{i}(\theta)$ as an expectation is not formally needed, but this interpretation as an average is convenient for subsequent sections.
\subsection{Centred base-$g$ expansion}\label{section centred exp}
We also require the notion of a \textit{centred} base-$g$ expansion of a real number, as used by Green in \cite{green2025waring}, where more detail regarding such expansions can be found. The centred base-$g$ expansion is closely linked to the regular base-$g$ expansion, but shifts the range of permissible digits. Let
\[\mathcal{R}_{g}\coloneq \{-\tfrac{g-1}{2},\ldots, \tfrac{g-1}{2}\}\textrm{ for odd }g \textrm{ and }\mathcal{R}_{g}\coloneq\{-\tfrac{g-2}{2},\ldots, \tfrac{g}{2}\} \textrm{ for even }g,\]
and let
\begin{equation}\label{Ig defn}
    I_{g}\coloneq (-1/2,1/2]\textrm{ for }g \textrm{ odd, and } I_{g}\coloneq\Big (-\frac{g-2}{2(g-1)},\frac{g}{2(g-1)}\Big ]\textrm{ for }g\textrm{ even.}
\end{equation}
Then for $\alpha \in I_{g}$, if 
\begin{equation}\label{centred exp defn}
    \alpha=\sum_{i\geq 1}\alpha_{i}g^{-i}, \textrm{ with }\alpha_{i}\in \mathcal{R}_{g}\textrm{ for all }i,
\end{equation}
we call this the centred base-$g$ expansion of $\alpha$. Note that $I_{g}$ is the interval for which the centred base-$g$ expansion of any element has no integer part, as opposed to $[0,1)$ for the regular base-$g$ expansion. Let $\mathcal{R}_{g}^{+}=\max \mathcal{R}_{g}$ and $\mathcal{R}_{g}^{-}=\min \mathcal{R}_{g}$.
As in the regular expansion, the centred expansion of a real number is unique, except when it ends in an infinite sequence of digits all equal to $\mathcal{R}^{+}$, or all $\mathcal{R}^{-}$. In this case, we would choose the latter representation. The reason for using this alternate notion of expansion is to use the following function, as defined in \cite{green2025waring}. \begin{defn}\label{wk defn}
    Let $w_{K}(\alpha)$ be the function counting the number of non-zero digits within the first $K$ digits of the centred base-$g$ expansion of $\alpha$, after the radix point. For $\alpha$ with expansion given in \cref{centred exp defn}, 
    \[w_{K}(\alpha)=\sum_{i=1}^{K}\ind_{\alpha_{i}\neq 0}.\]
\end{defn}
\section{Consequences of the local limit theorem}\label{section Psi}
In this section, we state a required local limit theorem and use this along with the probabilistic digit model outlined in \cref{section prob model} to prove results required for \cref{Section: major arc} and \cref{Section: minor arc}.

The local limit theorem we use is a special case of a more general local limit theorem, such as Theorem 13 of \cite[Ch. VII]{petrov1972independent}. Let $Y$ be a random variable uniformly taking values in $\{0,\ldots, g-1\}-(g-1)/2$, and let $T\geq 0$ be an integer. For $\nu\in \{0,\ldots, T(g-1)\}-\mu_{T}$, let $P(T,\nu)$ denote the probability that $T$ i.i.d. copies of $Y$ sum to $\nu$. 
\begin{cor}\label{Local limit theorem application version}
    For $T,\nu$ and $P(T,\nu)$ as defined above,
    \[P(T,\nu)=\frac{e^{-x^{2}/2}}{\sqrt{2\pi\sigma^{2}T}}+O_{g}(T^{-3/2}),\]
    where $x=\nu/\sqrt{\sigma^{2}T}$, and $\sigma^{2}=(g^{2}-1)/12$ is the variance of $X$. In particular, if $|x|<1/2$,
    \[P(T,\nu)=\frac{1}{\sqrt{2\pi\sigma^{2}T}}+O_{g}(\max (x^{2}T^{-1/2},T^{-3/2})).\]
\end{cor}
\begin{proof}
    The first statement is a corollary of Theorem 13 of \citep[Ch. VII]{petrov1972independent}, and we give a self-contained proof of this in \cref{Section: LLT}. The second statement follows immediately from expanding the exponential term.
\end{proof}
The second part of \cref{Local limit theorem application version} gives the bound on the sets $|S_{i}|$ claimed in \cref{S set bound}. Recall that $\mu_{K}=(g-1)K/2$ is the average digit sum for an integer $n\in [0,g^{K})$.
\begin{proof}[Proof of \cref{S set bound}]
    Note that $|S_{i}|=g^{K}P(K,k_{i}-\mu_{K})$, therefore this is an immediate consequence of the second part of \cref{Local limit theorem application version}, using \cref{k cond close to av} to show $x\coloneq (k_{i}-\mu_{K})/\sqrt{\sigma^{2}K}$ satisfies $|x|<1/2$ for sufficiently large $K$. 
\end{proof}
We now state some technical lemmas required for \cref{Section: minor arc}. These concern a function which we define below, which generalises the probability $P(K,\nu)=\prob(X_{0}+\ldots+X_{K-1}=\nu )$ to include powers of the $(g-1)$\sth roots of unity weighted by subsets of the random variables $X_{i}$. Here, the random variables $X_{i}$ are i.i.d. copies of the random variable $Y$, as defined in \cref{section prob model}.. For the rest of this section we assume that $g\geq 3$. 

Let $\mathbf{a}=(a_{0},\ldots, a_{g-2})$ be a $(g-1)$-tuple of non-negative integers. For any $\nu\in \{0,\ldots, (g-1)K\}-\mu_{K}$, the function $\Psi(\mathbf{a};\nu)$ is defined to be:
\begin{equation}\label{psi defn}
    \Psi(\mathbf{a};\nu)\coloneq \sum_{\substack{j_{1},\ldots, j_{g-2}\\ \sum_{i=0}^{g-2} j_{i}=\nu}}\prod_{s=0}^{g-2}\ru{sj_{s}}P(a_{s},j_{s}).
\end{equation}
where the sum ranges over all tuples $(j_{1},\ldots, j_{g-2})$ such that $P(a_{0},\nu-\sum_{r=1}^{g-2}j_{r})\prod_{s=1}^{g-2}P(a_{s},j_{s})>0$. Note that $P(a_{s},j_{s})>0$ for $j_{s}\in \{0,\ldots, a_{s}(g-1)\}-a_{s}(g-1)/2$, and $P(a_{s},j_{s})=0$ otherwise, so certainly we have the bound
\begin{equation}\label{j range remark}
|j_{s}|\leq a_{s}(g-1)/2 \textrm{ for }s=1,\ldots, g-2.
\end{equation}
Recall from \cref{section prob model} that $X_{0},\ldots, X_{K-1}$ are i.i.d. copies of $Y$, and that $\xb\coloneq \sum_{i=0}^{K-1}X_{i}g^{i}$. Then we have the relation
\begin{equation}\label{ru sum xb equiv}\ru{\sx{\xb}}=\ru{\xb}\end{equation}
 as $\xb=\sum_{j=0}^{K-1}X_{j}g^{j}\equiv \sum_{j=0}^{K-1}X_{j}\md{g-1}$. We use this fact to state the following equivalent expression for $\Psi(\mathbf{a};\nu)$. This is the form in which the function actually arises in calculations in \cref{Section: minor arc}, however the form stated in \cref{psi defn} is more convenient for the results in this section.
\begin{lemma}\label{Psi expectation form} Let $\mathbf{a}=(a_{0},\ldots, a_{g-2})$ for integers $a_{i}\geq 0$. For $j\in \{0,\ldots, g-2\}$ and $0\leq i\leq a_{j}-1$, let $Y_{j}\coloneq \sum_{i=0}^{a_{j}-1}Y_{j,i}g^{i}$, where the $Y_{j,i}$ are i.i.d. copies of $Y$. Then
 \[\Psi(\mathbf{a};\nu)=\E_{Y_{0},\ldots, Y_{g-2}}\ru{Y_{1}+\ldots+(g-2)Y_{g-2}}\ind_{ \sum_{j=0}^{g-2}\sx{Y_{j}}=\nu}.   \]
\end{lemma}
\begin{proof}
From \cref{ru sum xb equiv}, we have
    \begin{align}\label{psi equiv form 1}
        \E_{Y_{0},\ldots, Y_{g-2}}&\ru{Y_{1}+\ldots+(g-2)Y_{g-2}}\ind_{ \sum_{j=0}^{g-2}\sx{Y_{j}}=\nu}\nonumber\\&= \E_{Y_{0},\ldots, Y_{g-2}}\ru{\sx{Y_{1}}+\ldots+(g-2)\sx{Y_{g-2}}}\ind_{ \sum_{j=0}^{g-2}\sx{Y_{j}}=\nu}. 
    \end{align}
    Let $t_{j}$ denote the possible values of $\sx{Y_{j}}$ for $j=1,\ldots, g-2$. Then \cref{psi equiv form 1} equals
    \begin{align*}
        \E_{Y_{0}}\sum_{t_{1},\ldots,t_{g-2}}\prod_{j=1}^{g-2}\ru{jt_{j}}&\prob(\sx{Y_{j}}=t_{j})\ind_{\sx{Y_{0}}=\nu-\sum_{j=1}^{g-2}\sx{Y_{j}}}\\
        &=\sum_{t_{1},\ldots,t_{g-2}}P(a_{0},\nu-\sum_{j=1}^{g-2} t_{j})\prod_{j=1}^{g-2}\ru{jt_{j}}P(a_{j},t_{j})=\Psi(\mathbf{a};\nu).\qedhere
    \end{align*}
\end{proof}

Note that if $\mathbf{a}=(t,0,\ldots ,0)$, then $\Psi(\mathbf{a};\nu)=P(t,\nu)$ which can be estimated by \cref{Local limit theorem application version}. The following lemma generalises the second part of \cref{Local limit theorem application version}: it says that $\Psi(\mathbf{a};\nu)$ is approximately constant as $\nu$ varies, provided that $\nu$ is sufficiently small.
\begin{lemma}\label{psi lemma}
    Suppose that $a_{0},\ldots, a_{g-2}\geq 0$ be integers with $\mathbf{a}=(a_{0},\ldots, a_{g-2})$ and let $\nu\in \Z$. Suppose further that $|\nu|\leq Ca_{0}^{1/4}$ for some $C> 0$ and $a_{0}\geq a_{s}$ for $s=1,\ldots, g-2$. Then 
    \[\Psi(\mathbf{a};\nu)=\Psi(\mathbf{a};0)+O_{C,g}(\nu^{2}a_{0}^{-3/2}).\]
\end{lemma}
\begin{proof}
From the definition \cref{psi defn}, we see that
\begin{equation}\label{psi defn restated}
    \Psi(\mathbf{a};\nu)=\sum_{j_{1},\ldots, j_{g-2}}P(a_{0},\nu-\sum_{r=1}^{g-1}j_{r})\prod_{s=1}^{g-2}\ru{sj_{s}}P(a_{s},j_{s}).
\end{equation}
    First, we use the local limit theorem to estimate the term $P(a_{0},\nu-\sum_{r=1}^{g-2}j_{r})$. Let $J\coloneqq \sum_{r=1}^{g-2}j_{r}$. Note that $|J|< g^{2}\max a_{i}$ from \cref{j range remark}. From \cref{Local limit theorem application version}, we have 
    \begin{equation}\label{psi 0}
        P(a_{0},\nu-J)=(2\pi\sigma^{2}a_{0})^{-1/2}e^{-(\nu-J)^{2}/2\sigma^{2}a_{0}}+O_{g}(a_{0}^{-3/2}).
    \end{equation}
    Substituting this into \cref{psi defn restated} gives
    \begin{equation}\label{psi 1}
        \Psi(\mathbf{a};\nu)=(\sigma\sqrt{2\pi a_{0}})^{-1}e^{-\nu^{2}/2\sigma^{2}a_{0}}\sum_{j_{1},\ldots, j_{g-2}}e^{(2\nu J-J^{2})/2\sigma^{2}a_{0}}\prod_{s=1}^{g-2}\ru{sj_{s}}P(a_{s},j_{s})+O_{g}(a_{0}^{-3/2}),
    \end{equation}
    where the error term in \cref{psi 1} comes from that of \cref{psi 0}, and the fact that
    \begin{equation}\label{product of probs}
        \sum_{j_{1},\ldots, j_{g-2}}\prod_{s=1}^{g-2}P(a_{s},j_{s})=1.
    \end{equation}
By expanding the term $e^{-\nu^{2}/2\sigma^{2}a_{0}}=1+O_{C,g}(\nu^{2}a_{0}^{-1})$ in \cref{psi 1}, we will show that:
    \begin{claim}\label{psi 2 claim}
        \begin{equation}\label{psi 2 equation}
            \Psi(\mathbf{a};\nu)=(\sigma\sqrt{2\pi a_{0}})^{-1}\sum_{j_{1},\ldots, j_{g-2}}e^{(2\nu J-J^{2})/2\sigma^{2}a_{0}}\prod_{s=1}^{g-2}\ru{sj_{s}}P(a_{s},j_{s})+O_{C,g}(\nu^{2}a_{0}^{-3/2}).
        \end{equation}
    \end{claim} To establish the claim, we first show that $e^{(2\nu J-J^{2})/2\sigma^{2}a_{0}}\ll_{C,g}1$ by considering the ranges $|\nu|\leq |J|/2$ and $|\nu|>|J|/2$ separately. In the former case, when $|\nu|\leq |J|/2$,
    \[e^{(2\nu J-J^{2})/2\sigma^{2}a_{0}}\leq e^{(2|\nu J|-J^{2})/2\sigma^{2}a_{0}}\leq 1.\]
    In the latter case, when $|\nu|>|J|/2$, we use the assumption that $|\nu|\leq Ca_{0}^{1/4}$ to obtain
    \[e^{(2\nu J-J^{2})/2\sigma^{2}a_{0}}\leq e^{2\nu^{2}/\sigma^{2}a_{0}}\leq e^{2C^{2}/\sigma^{2}a_{0}^{1/2}}\ll_{C,g} 1.\]
    Thus using that $e^{(2\nu J-J^{2})/2\sigma^{2}a_{0}}\ll_{C,g}1$ and \cref{product of probs} gives
    \[\sum_{j_{1},\ldots, j_{g-2}}e^{(2\nu J-J^{2})/2\sigma^{2}a_{0}}\prod_{s=1}^{g-2}P(a_{s},j_{s})\ll_{C,g}1.\]
    This concludes the proof of \cref{psi 2 claim}.

    Now we discard the terms with $|J|\geq \sigma^{2}a_{0}/2|\nu|$ from \cref{psi 2 equation} (if there are any), showing that the contribution from these terms is negligible. Indeed, as $\sigma^{2}a_{0}/2|\nu|\leq |J|\leq g^{2}\max a_{s}$, and $\max a_{s}\leq a_{0}$,
    \begin{equation}\label{psi exp bound}
        e^{(2\nu J-J^{2})/2\sigma^{2}a_{0}}\leq e^{g^{2}|\nu|\max a_{s}/\sigma^{2}a_{0} -\sigma^{2}a_{0}/8\nu^{2}}\leq e^{Cg^{2}a_{0}^{1/4}/\sigma^{2}-\sigma^{2}a_{0}^{1/2}/8C^{2}}\ll_{C,g} a_{0}^{-10}.
    \end{equation}
    This additionally uses the assumption that $|\nu|\leq Ca_{0}^{1/4}$ for some $C>0$. Bounding the term $|\ru{sj_{s}}P(a_{s},j_{s})|\leq 1$ for all $s, j_{s}$ and using \cref{psi exp bound}, we have
    \[\Big | \sum_{\substack{j_{1},\ldots, j_{g-2}\\ \sigma^{2}a_{0}/2|\nu|\leq|J|\leq g^{2}a_{0}}} e^{(2\nu J-J^{2})/2\sigma^{2}a_{0}}\prod_{s=1}^{g-2}\ru{sj_{s}}P(a_{s},j_{s})\Big |\ll_{C,g} a_{0}^{-9},\]
    hence the contribution to \cref{psi 2 equation} from $(j_{1},\ldots, j_{g-2})$ such that $|J|$ is large can be absorbed into the overall error term of $O_{C,g}(\nu^{2}a_{0}^{-3/2})$. It remains to estimate the contribution to \cref{psi 2 equation} from $(j_{1},\ldots, j_{g-2}) $ with $|J|\leq \sigma^{2}a_{0}/2|\nu|$. To this end, we expand the term $e^{\nu J/\sigma^{2}a_{0}}$ in \cref{psi 2 equation}, giving
    \begin{equation}\label{psi proof eq}
        \Psi(\mathbf{a};\nu)=\frac{1}{\sigma\sqrt{2\pi a_{0}}}\sum_{\substack{j_{1},\ldots, j_{g-2}\\|J|\leq \sigma^{2}a_{0}/2|\nu|}}e^{-J^{2}/2\sigma^{2}a_{0}}\Big (1+\frac{\nu J}{\sigma^{2}a_{0}}+O_{g}\Big(\frac{\nu^{2}J^{2}}{a_{0}^{2}}\Big)\Big )\prod_{s=1}^{g-2}\ru{sj_{s}}P(a_{s},j_{s})+O_{C,g}(\nu^{2}a_{0}^{-3/2}).
    \end{equation}
    First, we show that
    \begin{equation}\label{psi main term}
        \frac{1}{\sigma\sqrt{2\pi a_{0}}}\sum_{\substack{j_{1},\ldots, j_{g-2}\\|J|\leq \sigma^{2}a_{0}/2|\nu|}}e^{-J^{2}/2\sigma^{2}a_{0}}\prod_{s=1}^{g-2}\ru{sj_{s}}P(a_{s},j_{s})=\Psi(\mathbf{a};0)+O_{C,g}(a_{0}^{-3/2}).
    \end{equation}
    To do so, note we can undo the truncation on the range of summation in \cref{psi main term}. If the contribution from the range $\sigma^{2}a_{0}/2|\nu|\leq |J|\leq g^{2}a_{0}$ is non-zero, then in this range:
    \[e^{-J^{2}/2\sigma^{2}a_{0}}\leq e^{-\sigma^{2}a_{0}/8\nu^{2}}\leq e^{-\sigma^{2}a_{0}^{1/2}/8C^{2}}\ll_{C,g}a_{0}^{-10}.\]
    Using this and \cref{product of probs} we have
    \begin{equation}\label{range truncation}
        \sum_{\substack{j_{1},\ldots, j_{g-2}\\|J|> \sigma^{2}a_{0}/2|\nu|}}e^{-J^{2}/2\sigma^{2}a_{0}}\prod_{s=1}^{g-2}\ru{sj_{s}}P(a_{s},j_{s})\ll_{C,g} a_{0}^{-10}.
    \end{equation}
  
    We now work in the full range of summation for $(j_{1},\ldots, j_{g-2})$. We apply \cref{Local limit theorem application version} to note that $e^{-J^{2}/2\sigma^{2}a_{0}}/\sigma\sqrt{2\pi a_{0}}=P(a_{0},-J)+O_{g}(a_{0}^{-3/2})$, hence
    \begin{align*}
         \frac{1}{\sigma\sqrt{2\pi a_{0}}}&\sum_{j_{1},\ldots, j_{g-2}}e^{-J^{2}/2\sigma^{2}a_{0}}\prod_{s=1}^{g-2}\ru{sj_{s}}P(a_{s},j_{s})\\&=\sum_{j_{1},\ldots, j_{g-2}}P(a_{0},-J)\prod_{s=1}^{g-2}\ru{sj_{s}}P(a_{s},j_{s})+O_{C,g}(a_{0}^{-3/2})\\
         &=\Psi(\mathbf{a};0)+O_{C,g}(a_{0}^{-3/2}),
    \end{align*}
    using the definition in \cref{psi defn} to obtain the second equality. This establishes \cref{psi main term}, giving the main term in \cref{psi lemma}.

   We now show that the remaining terms in \cref{psi proof eq} contribute only to the error term in \cref{psi lemma}. The $O_{g}(\nu^{2}J^{2}a_{0}^{-2})$ term within the summation over $(j_{1},\ldots, j_{g-2})$ in \cref{psi proof eq} can be absorbed into the error term $O_{C,g}(\nu^{2}a_{0}^{-3/2})$. Indeed, as $e^{-J^{2}/2\sigma^{2}a_{0}}J^{2}\leq 2\sigma^{2}a_{0}/e$, the contribution from this term is bounded as follows:
    \[\frac{\nu^{2}}{a_{0}^{5/2}}\sum_{\substack{j_{1},\ldots, j_{g-2}\\|J|\leq \sigma^{2}a_{0}/2|\nu|}}\prod_{s=1}^{g-2}P(a_{s},j_{s})e^{-J^{2}/2\sigma^{2}a_{0}}J^{2}\ll_{g} \frac{\nu^{2}}{a_{0}^{3/2}}\sum_{\substack{j_{1},\ldots, j_{g-2}\\|J|\leq \sigma^{2}a_{0}/2|\nu|}}\prod_{s=1}^{g-2}P(a_{s},j_{s})\ll_{g} \frac{\nu^{2}}{a_{0}^{3/2}},\]
    additionally using \cref{product of probs} in the final inequality. 

    We have shown that
    \begin{equation}\label{psi proof eq 2}
        \Psi(\mathbf{a};\nu)=\Psi(\mathbf{a};0)+\frac{\nu}{\sigma^{3}a_{0}^{3/2}\sqrt{2\pi}}\sum_{j_{1},\ldots, j_{g-2}}e^{-J^{2}/2\sigma^{2}a_{0}}J\prod_{s=1}^{g-2}\ru{sj_{s}}P(a_{s},j_{s})+O_{C,g}(\nu^{2}a_{0}^{-3/2}).
    \end{equation}
    Finally we show that \cref{psi proof eq 2} implies the statement of the lemma. It suffices to prove
    \begin{equation}\label{psi lemma final}
        \sum_{j_{1},\ldots, j_{g-2}}e^{-J^{2}/2\sigma^{2}a_{0}}J\prod_{s=1}^{g-2}\ru{sj_{s}}P(a_{s},j_{s})\ll_{C,g} 1.
    \end{equation}
Before proving \cref{psi lemma final}, we remark that the proof is essentially trivial in base 3. In this case, the equation on the left hand side of \cref{psi lemma final} equals
\[\sum_{j}(-1)^{j}je^{-j^{2}/2\sigma^{2}a_{0}}P(a_{1},j)=\sum_{j}g(j)=0,\]
as $g(j)\coloneq (-1)^{j}je^{-j^{2}/2\sigma^{2}a_{0}}P(a_{1},j)$ is an odd function. For $g\geq 4$, the proof of \cref{psi lemma final} is rather more involved, and we make use of the cancellation coming from the $\ru{sj_{s}}$ terms instead of the sign of $J$. Let
\begin{equation}\label{g defn}
        g(j_{1},\ldots, j_{g-2})\coloneq Je^{-J^{2}/2\sigma^{2}a_{0}}\prod_{s=1}^{g-2}P(a_{s},j_{s}).
    \end{equation}
We first work with the tuples $(j_{1},\ldots, j_{g-2})$ for which the following condition holds, in addition to the assumption throughout that $|j_{s}|\leq a_{s}(g-1)/2$. Suppose that $J=j_{1}+\ldots+j_{g-2}$ is such that
\begin{equation}\label{J tuples}
    \Big|\frac{2Jr+r^{2}}{2\sigma^{2}a_{0}}\Big |\leq 1/2\textrm{ for all  }r\in \{0,\ldots, g-2\}.
\end{equation}
For such $(j_{1},\ldots, j_{g-2})$, the function $g(j_{1},\ldots, j_{g-2})$ doesn't vary too much when incrementing $j_{1}$ by a small amount, as shown in the following claim.
\begin{claim}\label{g claim}
    For $(j_{1},\ldots, j_{g-2})$ such that \cref{J tuples} holds, and for all $r=0,\ldots, g-2$,
    \[|g(j_{1},\ldots, j_{g-2})-g(j_{1}+r,j_{2},\ldots, j_{g-2})|\ll_{g} e^{-J^{2}/2\sigma^{2}a_{0}}\Big (\frac{J^{2}}{a_{0}}(P(a_{1},j_{1})+a_{1}^{-3/2})+P(a_{1},j_{1})\Big )\prod_{s=2}^{g-2}P(a_{s},j_{s}).\]
\end{claim}

\begin{proof}[Proof of \cref{g claim}]
First we note that the claim is trivial when $J=0$. In this case, $(j_{1},\ldots, j_{g-2})=(0,\ldots, 0)$ and $g(0,\ldots, 0)=0$, and it follows from \cref{g defn} that $g(r,0,\ldots, 0)\ll_{g} \prod_{s=1}^{g-1}P(a_{s},j_{s})$. Thus from now on, we assume that $|J|\geq 1/2$. From \cref{Local limit theorem application version},
\begin{equation}\label{prob +r roughly const}
    |P(a_{1},j_{1})-P(a_{1},j_{1}+r)|\ll_{g}a_{1}^{-1/2}|e^{-a_{1}^{2}/2\sigma^{2}a_{1}}-e^{-(a_{1}+r)^{2}/2\sigma^{2}a_{1}}|+O_{g}(a_{1}^{-3/2}) \ll_{g} a_{1}^{-3/2}.
\end{equation}
Here, we also use that the function $e^{-x^{2}}$ is Lipschitz to obtain the final inequality. We also have, under the assumptions of \cref{J term suff cancellation claim}, 
\begin{equation}\label{exp Jr term}
    \exp\Big (-\frac{2Jr+r^{2}}{2\sigma^{2}a_{0}} \Big )=1+O_{g}\Big (\frac{J}{a_{0}}\Big ).
\end{equation}
From the definition in \cref{g defn}, 
\begin{align*}
    & |g(j_{1},\ldots, j_{g-2})-g(j_{1}+r,j_{2},\ldots, j_{g-2})|\\
     &=e^{-J^{2}/2\sigma^{2}a_{0}}\prod_{s=2}^{g-2}P(a_{s},j_{s})|JP(a_{1},j_{1})-(J+r)e^{-(2Jr+r^{2})/2\sigma^{2}a_{0}}P(a_{1},j_{1}+r)|.
\end{align*}
Expanding the term $e^{-(2Jr+r^{2})/2\sigma^{2}a_{0}}$ using \cref{exp Jr term}, and using \cref{prob +r roughly const} to estimate $P(a_{1}+r,j_{1})$, we obtain,
\begin{align*}
    |JP(a_{1},j_{1})-(J+r)e^{-(2Jr+r^{2})/2\sigma^{2}a_{0}}&P(a_{1},j_{1}+r)|\\&=|JP(a_{1},j_{1})-(J+r)(1+O_{g}(J/a_{0}))(P(a_{1},j_{1})+O_{g}(a_{1}^{-3/2})))|\\
    &\ll_{g}\frac{J^{2}}{a_{0}}(P(a_{1},j_{1})+a_{1}^{-3/2})+P(a_{1},j_{1}).
\end{align*}
To simplify the final expression, we have used that $r\ll_{g} 1$, and that for any tuple $(j_{1},\ldots, j_{g-2})$, $J^{2}\geq |J|/2$. 
\end{proof}
We use \cref{g claim} to prove the following:
\begin{equation}\label{J term suff cancellation claim}
     \sum_{j_{1},\ldots, j_{g-2}}e^{-J^{2}/2\sigma^{2}a_{0}}J\prod_{s=1}^{g-2}\ru{sj_{s}}P(a_{s},j_{s})\ll_{g} 1.
\end{equation} 
We have
\begin{equation*}
    \sum_{j_{1},\ldots, j_{g-2}}e^{-J^{2}/2\sigma^{2}a_{0}}J\prod_{s=1}^{g-2}\ru{sj_{s}}P(a_{s},j_{s})=\sum_{j_{1},\ldots, j_{g-2}}g(j_{1},\ldots, j_{g-2})\prod_{s=1}^{g-2}\ru{sj_{s}}
\end{equation*}
To get the cancellation required, we use \cref{g claim} to assert that $g(j_{1}+r,\ldots, j_{g-2})$ is essentially constant as $r$ varies in $\{0,\ldots, g-2\}$, which allows us to get cancellation from the $\prod_{s=1}^{g-2}\ru{sj_{s}}$ term. Our aim is to split the range of summation into sets where $j_{1}$ has a fixed congruence modulo $g-1$. Note that if $g$ is even and $a_{1}$ odd, the range of $j_{1}$ is not contained in the integers, as $j_{1} \in \{0,\ldots, a_{1}(g-1)\}-a_{1}(g-1)/2$. In this case, $j_{1}+1/2\subset \Z$, so we can run the following argument by multiplying through by a factor of $\ru{1/2}$. If $a_{1}(g-1)/2\in \Z$, let $x\coloneq -a_{1}(g-1)/2$, otherwise let $x\coloneq -a_{1}(g-1)/2-1/2$. Note that from \cref{j range remark}, the range of $j_{1}$ is a multiple of $g-1$. Thus splitting the range of $j_{1}$ (compensating by a factor of $\ru{1/2}$ if necessary), 
\begin{align}
    \sum_{j_{1},\ldots, j_{g-2}}g(j_{1},\ldots, j_{g-2})\prod_{s=1}^{g-2}\ru{sj_{s}}=\sum_{\substack{j_{1},\ldots, j_{g-2}\\ j_{1}\equiv x \md{g-1}}}\prod_{s=1}^{g-2}\ru{sj_{s}}\sum_{r=0}^{g-2}\ru{r}g(j_{1}+r,j_{2},\ldots, j_{g-2})\nonumber\\
    =\sum_{\substack{j_{1},\ldots, j_{g-2}\\ j_{1}\equiv x \md{g-1}}}\prod_{s=1}^{g-2}\ru{sj_{s}}\sum_{r=0}^{g-2}\ru{r}\Big ( g(j_{1},j_{2},\ldots, j_{g-2})+O_{g}(E(J,a_{0},a_{1}))\Big )\label{final g sum}
\end{align}
from \cref{J term suff cancellation claim}, where
\[E(J,a_{0},a_{1})=e^{-J^{2}/2\sigma^{2}a_{0}}\Big (\frac{J^{2}}{a_{0}}(P(a_{1},j_{1})+a_{1}^{-3/2})+P(a_{1},j_{1})\Big )\prod_{s=2}^{g-2}P(a_{s},j_{s}).\]
The first term in the sum over $r$ is zero:
\begin{align*}
    &\sum_{\substack{j_{1},\ldots, j_{g-2}\\ j_{1}\equiv x \md{g-1}}}\prod_{s=1}^{g-2}\ru{sj_{s}}\sum_{r=0}^{g-2}\ru{r}g(j_{1},j_{2},\ldots, j_{g-2})\\&=\sum_{\substack{j_{1},\ldots, j_{g-2}\\ j_{1}\equiv x \md{g-1}}}\prod_{s=1}^{g-2}\ru{sj_{s}}g(j_{1},j_{2},\ldots, j_{g-2})\sum_{r=0}^{g-2}\ru{r}=0.
\end{align*}
Finally, we show that the error term in \cref{final g sum} is $\ll_{g} 1$, that is, we show
\begin{equation}\label{final g sum error}
    \sum_{j_{1},\ldots, j_{g-2}}E(J,a_{0},a_{1})\ll_{g} 1.
\end{equation}
First note that from \cref{product of probs}, 
\begin{equation}\label{no J sq term}
    \sum_{j_{1},\ldots, j_{g-2}}e^{-J^{2}/2\sigma^{2}a_{0}}\prod_{s=1}^{g-2}P(a_{s},j_{s})\ll_{g} 1.
\end{equation}
and further using that $\sup J^{2}e^{-J^{2}/2\sigma^{2}a_{0}}\ll_{g} a_{0}$ gives
\begin{equation}\label{J sq 1}
    \frac{1}{a_{0}}\sum_{j_{1},\ldots, j_{g-2}}J^{2}e^{-J^{2}/2\sigma^{2}a_{0}}\prod_{s=1}^{g-2}P(a_{s},j_{s})\ll_{g} 1.
\end{equation}
    Moreover, a similar statement holds when replacing $P(a_{1},j_{1})$ by $a_{1}^{-3/2}$:
    \begin{align}
         \frac{1}{a_{0}a_{1}^{3/2}}\sum_{j_{1},\ldots, j_{g-2}}J^{2}e^{-J^{2}/2\sigma^{2}a_{0}}\prod_{s=2}^{g-2}P(a_{s},j_{s})&\ll_{g} \sum_{j_{1}}\frac{1}{a_{1}^{3/2}}\sum_{j_{2},\ldots, g-2}\prod_{s=2}^{g-2}P(a_{s},j_{s})\nonumber\\
         &=\sum_{j_{1}}\frac{1}{a_{1}^{3/2}}\ll_{g} \frac{1}{a_{1}^{1/2}},\label{J sq 2}
    \end{align}
    where the last line uses that $j_{1}$ ranges over $|j_{1}|\leq a_{1}(g-1)/2$.
    Combining \cref{no J sq term}, \cref{J sq 1} and \cref{J sq 2} gives \cref{final g sum error}.

Assuming \cref{J term suff cancellation claim}, it suffices to show that the contribution from $(j_{1},\ldots, j_{g-2})$ such that \cref{J tuples} doesn't hold is bounded. Suppose for $(j_{1},\ldots, j_{g-2})$ and $J=j_{1}+\ldots +j_{g-2}$,
\[\Big|\frac{2Jr+r^{2}}{2\sigma^{2}a_{0}}\Big |>\frac{1}{2}.\]
In particular, for such $J$, $|J|>\sigma^{2}a_{0}/2-(g-1)/2$, therefore
\begin{equation*}
    J^{2}e^{-J^{2}/2\sigma^{2}a_{0}}\ll_{g} a_{0}^{2}e^{-\sigma^{2}a_{0}/8}\ll_{g} a_{0}^{-10}.
\end{equation*}
Thus the contribution from $j_{1},\ldots, j_{g-2}$ where $|J|$ is large is
\begin{align*}
    \Big |\sum_{\substack{j_{1},\ldots, j_{g-2}\\ |J|\geq \sigma^{2}a_{0}/4}}J^{2}e^{-J^{2}/2\sigma^{2}a_{0}}\prod_{s=1}^{g-2}P(a_{s},j_{s})\Big |&\ll_{g} a_{0}^{-10}.\qedhere
\end{align*}
\end{proof}
\cref{psi lemma} can be used to bound $\Psi(\mathbf{a};0)$ for certain $\mathbf{a}$. 
\begin{cor}\label{psi 0 bound}
    Suppose $\mathbf{a}=(a_{0},\ldots, a_{g-2})$ is such that $a_{i}\geq 0$ and $a_{0}\geq a_{i}$ for $0\leq i\leq g-2$. Furthermore, let $0\leq m\leq Ca_{0}^{1/4}$ be an integer, for some $C>0$. Let $1\leq t \leq g-2$. For $\mathbf{a}'=(a_{0},\ldots, a_{t}+m,\ldots,a_{g-2})$, we have
    \[\Psi(\mathbf{a}';0)\leq \frac{1}{g^{m}}+O_{C,g}(m^{2}a_{0}^{-3/2}).\]
\end{cor}
\begin{proof}
Let $W_{j,i}, $ for $ 0\leq j\leq g-2, 0\leq i\leq a_{j}-1 $, and $Z_{j}, $ for $0\leq j\leq m-1$, be i.i.d. copies of $Y$, and let $W_{j}\coloneq \sum_{i=0}^{a_{j}-1}W_{j,i}g^{i}$ and $Z\coloneq \sum_{i=0}^{m-1}Z_{i}g^{i}$. With the alternate definition of $\Psi(\mathbf{a}';0)$ as stated in \cref{Psi expectation form}, we have
    \begin{align*}
        \Psi(\mathbf{a}';0)&=\E_{W_{0},\ldots, W_{g-2},Z}\Ru{\sum_{r=1}^{g-2}r W_{r}+tZ}\ind_{\sum_{j=0}^{g-2}\sx{W_{j}}=-\sx{Z}}.\end{align*}
        Separating out the contribution from $Z$ and using the definition of $\Psi(\mathbf{a};-\sx{Z})$ from \cref{Psi expectation form},
\begin{align*}
      \Psi(\mathbf{a}';0)&=\E_{Z}\ru{tZ}\E_{W_{0},\ldots, W_{g-2}}\Ru{\sum_{r=1}^{g-2}rW_{r}}\ind_{\sum_{j=0}^{g-2}\sx{W_{j}}=-\sx{Z}}\\
      &=\E_{Z}\ru{tZ}\Psi(\mathbf{a};-\sx{Z}).
\end{align*}
As $|\sx{Z}|\ll_{g} Ca_{0}^{1/4}$, \cref{psi lemma} applies to give
    \[\Psi(\mathbf{a};-\sx{Z})=\Psi(\mathbf{a};0)+O_{C,g}(m^{2}a_{0}^{-3/2}).\]
    Hence
    \[ \Psi(\mathbf{a}';0)=\E_{Z}\ru{tZ}\Psi(\mathbf{a};0)+O_{C,g}(m^{2}a_{0}^{-3/2}).\]
From \cref{ru sum xb equiv}, that is, using that $\sx{Z}\equiv Z\md{g-1}$, and the fact that the $Z_{j}$ are i.i.d. copies of $Y$, we have
    \[\E_{Z}\ru{tZ}=\prod_{j=0}^{m-1}\E_{Z_{j}}\ru{tZ_{j}}=\Big (\E_{Y}\ru{tY}\Big )^{m}\ll \frac{1}{g^{m}}.\]
    Here we also use that $t\neq 0$ and that $Y+(g-1)/2$ uniformly takes values in $\{0,\ldots, g-1\}$ . Finally, using the fact that $|\Psi(\mathbf{a};0)|\leq 1$ for any tuple $\mathbf{a}$ gives the result.
\end{proof}
\section{Major arcs contribution}\label{Section: major arc} 
In this section, we establish \cref{major arc contrib}: that the contribution to \cref{main S count} from the major arcs $\mathfrak{M}$ gives the main term in \cref{Main theorem S count}. Recall that we have the following major arcs $\mathfrak{M}$,
\[\mathfrak{M}\coloneqq \bigcup_{j=0}^{g-2}\big[\frac{j}{g-1}-\varepsilon,\frac{j}{g-1}+\varepsilon\big ]\]
for $\varepsilon\coloneq K^{3/4}g^{-K}/(g-1)$. We also have that $f_{i}(\theta+j/(g-1))=\ru{jk_{i}}f_{i}(\theta)$ for $j\in \Z$ from \cref{cong fact}. Therefore in order to evaluate the contribution from the major arcs to \cref{main S count}, it suffices to consider the contribution from $\mathfrak{M}$ around 0, as
\begin{align}
    \int_{\mathfrak{M}}\prod_{i=1}^{3}f_{i}(\theta)e(-M\theta)\dif\theta&=\sum_{j=0}^{g-2}\int_{|\theta|\leq \varepsilon}\prod_{i=1}^{3}f_{i}(\theta+\tfrac{j}{g-1})e(-M(\theta+\tfrac{j}{g-1}))\dif\theta\nonumber \\
    &=\Big (\sum_{j=0}^{g-2}\ru{j(k_{1}+k_{2}+k_{3}-M)}\Big)\int_{|\theta|\leq \varepsilon}\prod_{i=1}^{3}f_{i}(\theta)e(-M\theta)\dif\theta \nonumber\\
    &=(g-1)\int_{|\theta|\leq \varepsilon}\prod_{i=1}^{3}f_{i}(\theta)e(-M\theta)\dif\theta \mbox{\qquad\qquad by \cref{k cond: congruences}.} \label{major arc near zero integral}
\end{align}
One can view the factor of $(g-1)$ in \cref{major arc near zero integral} as a very simple singular series, with the integral term being our singular integral. To evaluate this integral, we require the following lemma. This gives an asymptotic for $f_{i}(\theta)$ on a range around 0 which includes $[-\varepsilon,\varepsilon]$ as well as on some minor arc points; the asymptotic will be used to bound these points later. By combining \cref{major arc asymp lemma} with \cref{trans inv}, we can get an asymptotic for $f_{i}(\theta)$ for $\theta \in \mathfrak{M}$ more generally.
\begin{lemma}\label{major arc asymp lemma}
    Let $\ell\coloneqq \floor{C_{0}\log K}$ be as defined in \cref{l choice}. For all $\theta$ such that $\V{\theta}\leq g^{-K+2\ell^{2}}/(g-1)$,
    \[f_{i}(\theta)=\frac{g^{K}}{\sqrt{2\pi\sigma^{2}K}}\int_{0}^{1}e(g^{K}\theta x)\dif x+O_{g}\Big(\frac{g^{K}\ell^{4}}{K^{3/2}}\Big).\]
\end{lemma}
To prove \cref{major arc asymp lemma}, we switch to the probabilistic model for $n<g^{K}$ outlined in \cref{section prob model}. Recall that $X_{j}$, $j=0,\ldots, K-1$, are i.i.d. copies of the uniform random variable taking values in $\{0,\ldots, g-1\}-(g-1)/2$, with $\xb=\sum_{j=0}^{K_1}X_{j}g^{j}$. For an indexing set $\mathcal{S}\subseteq \{0,\ldots, K-1\}$, let $\xb_{\mathcal{S}}$ denote the random variable 
\begin{equation}\label{Xs defintion}
    \xb_{\mathcal{S}}\coloneq \sum_{j\in \mathcal{S}}X_{j}g^{j}.
    \end{equation}

\begin{proof}[Proof of \cref{major arc asymp lemma}]
We can rewrite $f_{i}(\theta)$ as the following expectation from \cref{lemma expectation version f}, 
 \begin{equation}\label{equation f exp major arc}    
f_{i}(\theta)=g^{K}e\Big(\frac{g^{K}-1}{2}\theta\Big)\mathbb{E}_{\xb}e(\xb\theta)1_{\sx{\xb}=\xi_{i}}\end{equation}
where $\xi_{i}\coloneq k_{i}-\mu_{K}$ is the distance of the target digit sum, $k_{i}$, from the average value. Note that the condition given in \cref{k cond close to av} implies that $|\xi_{i}|\ll_{g} 1$. 

As $|\theta|\leq g^{-K+2\ell^{2}}/(g-1)$, the value of $e(\xb\theta)$ is determined mainly by the value of the random variables $X_{K-2\ell ^{2}-\ell},\ldots, X_{K-1}$, up to a small error. Let \[\mathcal{D}\coloneq \{K-2\ell ^{2}-\ell,\ldots, K-1\} \textrm{ and }\mathcal{E}=\{0,\ldots, K-1\}\setminus \mathcal{D}.\] Then as $|\sum_{j\notin \mathcal{D}}X_{j}g^{j}\theta|\leq g^{-\ell}$, we have  $e(\xb\theta)=e(\xbd\theta)+O(g^{-\ell})$. Here, $\xbd$ is defined as in \cref{Xs defintion}. Let $L\coloneq |\mathcal{D}| = 2\ell^{2}+\ell$. Thus,
\begin{align*}    \E_{\xb}e(\xb\theta)\ind_{\sx{\xb}=\xi_{i}}&=\E_{\xbd}e(\xbd\theta)\E_{\xbee}\ind_{\sx{\xb}=\xi_{i}}+O(g^{-\ell})\\
    &=\E_{\xbd}e(\xbd\theta)P(K-L,\xi_{i}-\sx{\xbd})+O(g^{-\ell}),
    \end{align*} 
where the notation $P(T,t)$ is used to denote $\prob(X_{1}+\ldots +X_{T}=t)$. By \cref{Local limit theorem application version},
\begin{equation}\label{prob term 1}
    P(K-L,\xi_{i}-\sx{\xbd})=\frac{e^{-x^{2}/2}}{\sqrt{2\pi\sigma^{2}(K-L)}}+O_{g}((K-L)^{-3/2}),
\end{equation}
where $x=(\xi_{i}-\sx{\xbd})/\sqrt{\sigma^{2}(K-L)}$. Note that $x$ has size approximately $L/\sqrt{K}$, coming from the fact that $|\xi_{i}-\sx{\xbd}|\ll_{g} |\mathcal{D}|$ and $|\mathcal{D}|=L\ll l^{2}$. From expanding the exponential term in \cref{prob term 1} and using that $(K-L)^{-1/2}-K^{-1/2}=O(LK^{-3/2})$, we can remove the dependence of the values of $\xi_{i}$ and $\sx{\xbd}$ from $P(K-L,\xi_{i}-\sx{\xbd})$, giving
 \[P(K-L,\xi_{i}-\sx{\xbd})=\frac{1}{\sqrt{2\pi\sigma^{2}K}}+O_{g}(\ell^{4}K^{-3/2}).\]
    Hence the expectation over all digits in $\xb$ may be replaced by an average over the digits indexed by $\mathcal{D}$ only,   \begin{align}\label{expectation over D}\E_{\xb}e(\xb\theta)\ind_{\sx{\xb}=\xi_{i}}=\frac{1}{\sqrt{2\pi\sigma^{2}K}}\E_{\xbd}e(\xbd\theta)+O_{g}\big (\ell^{4}K^{-3/2}\big ).
    \end{align}
Note that $\xbd g^{-K+L}+(g^{L}-1)/2$ takes values in $\{0,\ldots, g^{L}-1\}$ uniformly at random, so we can write the expectation over $\xbd$ explicitly as a normalised sum,
\[\E_{\xbd}e(\xbd\theta)=\frac{1}{g^{L}}e(-\theta(g^{K}-1)/2)\sum_{j=0}^{g^{L}-1}e(g^{K-L}j\theta).\]
Evaluating this series and comparing it to the corresponding integral gives:
\begin{align*}
    \frac{1}{g^{L}}e(-\theta(g^{K}-g^{K-L})/2)\sum_{j=0}^{g^{L}-1}e(g^{K-L}j\theta)&=\frac{1}{g^{L}}e(-\theta(g^{K}-g^{K-L})/2)\int_{0}^{g^{L}}e(g^{K-L}\theta x)\dif x \\&=e(-\theta(g^{K}-g^{K-L})/2)\int_{0}^{1}e(g^{K}\theta x)\dif x.
\end{align*}
Combining this with \cref{expectation over D} and multiplying through by $g^{K}e(\theta(g^{K}-1)/2)$ gives
\begin{align*}
    f_{i}(\theta)&=e\Big ( \frac{g^{K-L}-1}{2}\theta\Big)\frac{g^{K}}{\sqrt{2\pi\sigma^{2}K}}\int_{0}^{1}e(g^{K}\theta x)\dif x+O_{g}\Big (\frac{g^{K}\ell^{4}}{K^{3/2}}\Big ).
\end{align*}
To obtain the expression for $f_{i}(\theta)$ given in the statement of the lemma, note that,
\[e\Big ( \frac{g^{K-L}-1}{2}\theta\Big)=1+O_{g}(g^{-\ell}).\]
This follows from the fact that $|\theta|\leq g^{-K+L-\ell}/(g-1)$. From the choice of $\ell$ given in \cref{l choice}, we have that $g^{K-\ell}=O_{g}(g^{K}K^{-3/2})$. 
\end{proof}
We now prove \cref{major arc contrib}. As \cref{major arc asymp lemma} holds for $|\theta|\leq g^{-K+2\ell^{2}}/2$, this includes $|\theta|\leq \varepsilon$ as $\varepsilon=K^{3/4}g^{-K}/(g-1)$ and $\ell \geq \log K$ from \cref{l choice}. Hence we can apply \cref{major arc asymp lemma} to estimate the contribution from the $f_{i}(\theta)$ for $|\theta|\leq \varepsilon$, 
\[\prod_{i=1}^{3}f_{i}(\theta)=g^{3K}(2\pi\sigma^{2}K)^{-3/2}\Big (\int_{0}^{1}e(g^{K}\theta x)\dif x\Big )^{3}+O_{g}(g^{3K}\ell^{4}K^{-5/2}).\]
Substituting into \cref{major arc near zero integral} and using the change of variable $\eta=g^{K}\theta$, 
\begin{align}
  (g-1)\int_{\mathfrak{M}}\prod_{i=1}^{3}f_{i}(\theta)&e(-M\theta)\dif \theta=\frac{(g-1)g^{3K}}{(2\pi\sigma^{2}K)^{3/2}}\int_{|\theta|\leq \varepsilon} \Big(\int_{0}^{1}e(g^{K}\theta x)\dif x\Big )^{3}e(-M\theta)\dif \theta+O_{g}\Big (\frac{g^{2K}\ell^{4}}{K^{7/4}}\Big )\nonumber\\
    &=\frac{(g-1)g^{2K}}{(2\pi\sigma^{2}K)^{3/2}}\int_{|\eta|\leq g^{K}\varepsilon}\Big (\int_{0}^{1}e(\eta x)\dif x \Big )^{3} e (-Mg^{-K}\eta )\dif \eta +O_{g}\Big (\frac{g^{2K}\ell^{4}}{K^{7/4}}\Big ).\label{major arc substitution}
\end{align}
To evaluate the integral, we follow the treatment of the singular integral in \citep[Ch.~4]{davenport2005analytic}, though it is considerably simpler than the case arising in Waring's problem. 
Firstly, we extend the range of integration of $\eta$ to $(-\infty,\infty)$. This accrues error 
\begin{align}\label{eq: ext integration range}
  \frac{g^{2K}}{K^{3/2}} \int_{|\eta|\geq g^{K}\varepsilon}\big (\int_{0}^{1}e(\eta x)\dif x\big )^{3}e(-\eta M g^{-K})\dif \eta &\ll  \frac{g^{2K}}{K^{3/2}}\int_{|\eta|\geq g^{K}\varepsilon} \Big(\frac{|e(\eta)-1|}{|\eta|}\Big)^{3} \dif \eta\ll \frac{1}{K^{3/2}\varepsilon^{2}}.
\end{align}
As $K^{3/2}\varepsilon^{2}=K^{3}g^{-2K}/(g-1)^{2}$, extending the range of integration contributes error $O_{g}(g^{2K}K^{-3})$. To evaluate the extended integral, note that
\begin{align}\label{ext integral to convol}
    \int_{-\infty}^{\infty}\Big (\int_{0}^{1}e(\eta x)\dif x\Big )^{3}e(-\eta M g^{-K})\dif \eta=\int_{-\infty}^{\infty} \widehat{h}(\eta)^{3}e(-\eta M g^{-K})\dif \eta = h \ast h \ast h(Mg^{-K})
\end{align}
where $h=\ind_{[0,1]}$. Using that $Mg^{-K}\in(1/g,1]$, we can explicitly calculate $h*h*h(Mg^{-K})$,
\begin{align}
    h*h*h(Mg^{-K})&=\int_{-\infty}^{\infty}\int_{0}^{1}\ind_{y\in[z-1,z]}\ind_{z\in[Mg^{-K}-1,Mg^{-K}]}\dif y\dif z\nonumber\\
    &=\int_{Mg^{-K}-1}^{Mg^{-K}}z\ind_{z\in [0,1)}+(2-z)\ind_{z\in [1,2)}\dif z=M^{2}g^{-2K}/2.\label{eval convol}
\end{align}
This concludes the proof of \cref{major arc contrib}; extending the integral in \cref{major arc substitution} and substituting \cref{ext integral to convol} and \cref{eval convol} into the extended integral gives the stated contribution from the major arcs.

\section{Minor arcs contribution}\label{Section: minor arc}
In this section, we prove \cref{minor arc bound}, that is, showing that $f_{i}(\theta)$ is uniformly bounded by $\ll g^{K}K^{-5/4}$ on the minor arcs $\mathfrak{m}$. For a subset of the minor arcs, existing results give a stronger bound, which we demonstrate shortly. The remaining minor arc points have a specific structure, which we exploit to show the required bound on $f_{i}(\theta)$. 

We have the following bound on $|f_{i}(\theta)|$ due to Fouvry and Mauduit \cite{fouvry2005entiers}.
\begin{theorem}[\cite{fouvry2005entiers}]\label{FM bound 2}
For $\theta \in \R/\Z$,
     \[|f_{i}(\theta)| \leq g^{K}\exp \Big (-\frac{1}{2g}\sum_{i=0}^{K-1}\Vert g^{i}(g-1)\theta\Vertr^{2}\Big ).\]
\end{theorem}
This theorem as written above is not stated explicitly in \cite{fouvry2005entiers}, rather it follows immediately from the proof of \citep[Theorem~1.2]{fouvry2005entiers}. We sketch this in \cref{section FM bound}.

\cref{FM bound 2} gives a strong saving over the bound required for \cref{minor arc bound} whenever $\theta $ is such that
\begin{equation}\label{sum frac parts condition}\sum_{i=1}^{K}\Vert (g-1)g^{i}\theta\Vert_{\R/\Z}^{2}\geq 2gC \log K,\end{equation}
for large enough $C$. Thus it remains to prove \cref{minor arc bound} for $\theta$ such that \cref{sum frac parts condition} does not hold for sufficiently large $C$. In order to do so, we need to understand the structure of such $\theta$, which we achieve by using the centred base-$g$ expansion of $(g-1)\theta$. Recall from \cref{section centred exp}, that for real $\alpha \in I_{g}$, the centred base-$g$ expansion of $\alpha$ is
\[\alpha=\sum_{i\geq 1}\varepsilon_{i}g^{-i}, \textrm{ with }\varepsilon_{i} \in \Big (-\frac{g}{2},\frac{g}{2}\Big ]\cap \Z \textrm{ for all }i.\]
Given the centred expansion of $\alpha$ above, we define 
\[w_{K}(\alpha)\coloneq \sum_{i=1}^{K}\ind_{\varepsilon_{i}\neq 0},\]
which counts the number of non-zero digits within the first $K$ digits of the centred expansion of $\alpha$. The following lemma due to Green allows us to replace the sum over fractional parts in \cref{sum frac parts condition} with the function $w_{K}$.
\begin{lemma}\citep[Lemma~7.2]{green2025waring}\label{lemma: frac part to few non-zero}
For $g\geq 3$ and $\alpha \in \R$,
    \[\frac{w_{K}(\alpha)}{16g^{2}}\leq \sum_{i=0}^{K-1}\Vert g^{i}\alpha\Vertr^{2}\leq w_{K}(\alpha).\]
\end{lemma}
As we define $\ell\coloneq\lceil 384g^{3}\log K \rceil$ in \cref{l choice}, we have the following corollary to \cref{FM bound 2}.
\begin{cor}\label{FM very good g}
    For $\ell$ as defined in \cref{l choice} and $\theta$ such that $w_{K}((g-1)\theta)>\ell$,
    \[f_{i}(\theta)\ll_{g} {g^{K}}{K^{-12}},\textrm{ for }i=1,2,3.\]
\end{cor}
\begin{proof} Recall that $\ell =\lceil 384g^{3}\log K \rceil$. From the lower bound in \cref{lemma: frac part to few non-zero},
    \begin{align*}
        \exp\Big (-\frac{1}{2g}\sum_{i=1}^{K}\V{(g-1)g^{i}\theta}^{2}\Big)&\leq \exp (-w_{K}((g-1)\theta)/32g^{3})\leq \exp(-\ell/32g^{3})\leq K^{-12}.
    \end{align*}
    Inserting this upper bound into the bound given by \cref{FM bound 2} gives the required bound for $f_{i}(\theta)$.
\end{proof}
\begin{remark}\label{base 2 remark}
    It is at this point that we have to restrict our results to base $g$, for $g\geq 3$. This is because \cref{lemma: frac part to few non-zero} is not valid for base 2; discussion regarding why this is not the case is given in \cite{green2025waring}. 
    
   As a consequence, we cannot use the base-2 version of $w_{K}(\alpha)$ to model the function $\sum_{i=1}^{K}\V{2^{i}\alpha}^{2}$. Instead of counting the number of non-zero digits within the first $K$ digits of the expansion of $\alpha$, we count the number of times consecutive digits alternate value. More precisely, if $\alpha$ has base-2 expansion $\alpha=\sum_{i\geq 1}\varepsilon_{i}2^{-i}$, let $d_{K}(\alpha)$ denote
    \[d_{K}(\alpha )\coloneq |\{(\varepsilon_{i},\varepsilon_{i+1})\st \varepsilon_{i}\neq \varepsilon_{i+1}, i\in \{1,\ldots, K\}\}|.\]
    One can show that $d_{K}(\alpha)\asymp \sum_{i=1}^{K}\V{2^{i}\alpha}^{2}$, and use this to establish the base-2 versions of our results by replacing instances of $w_{K}$ by $d_{K}$, and making some small technical adjustments.
\end{remark}
We will show that \cref{major arc asymp lemma} directly gives \cref{minor arc bound} for $\theta \in \mathfrak{m}$ with $\V{(g-1)\theta}\leq g^{-K+2\ell^{2}}$. Recall that $\mathfrak{m}\coloneq \{\theta\in \R/\Z \st\V{(g-1)\theta}> K^{3/4}g^{-K}\}$.
\begin{cor}\label{close to zero minor arc}
    Suppose that $K^{3/4}g^{-K}\leq \V{(g-1)\theta}\leq g^{-K+2\ell^{2}}$. Then
    \[f_{i}(\theta)\ll_{g} {g^{K}}{K^{-5/4}}.\]
\end{cor}
\begin{proof}
Suppose first that $K^{3/4}g^{-K}\leq \V{\theta}\leq g^{-K+2\ell^{2}}/(g-1)$. Then as $\V{(g-1)\theta}\leq g^{-K+2\ell^{2}}$, \cref{major arc asymp lemma} applies to give that
\begin{align*}
    f_{i}(\theta)&=\frac{g^{K}}{\sqrt{2\pi\sigma^{2}K}}\int_{0}^{1}e(g^{K}\theta x)\dif x+O_{g}\Big (\frac{g^{K}\ell^{4}}{K^{3/2}}\Big)\\
    &=\frac{1}{(2\pi)^{3/2}\sigma K^{1/2}i\theta}(e(g^{K}\theta)-1)+O_{g}\Big (\frac{g^{K}\ell^{4}}{K^{3/2}}\Big)\ll_{g} \frac{g^{K}}{K^{5/4}}.
\end{align*}
Now suppose that $\V{(g-1)\theta}\leq g^{-K+2\ell^{2}}$, but that $\V{\theta}>g^{-K+2\ell^{2}}/(g-1)$. In this case, there exists $j\in \{1,\ldots, g-2\}$ such that $K^{3/4}g^{-K} \leq \V{\theta-j/(g-1)}\leq g^{-K+2\ell^{2}}/(g-1)$. Thus \cref{major arc asymp lemma} gives, following the above argument,
\[f_{i}\Big (\theta-\frac{j}{g-1}\Big)\ll_{g} \frac{g^{K}}{K^{5/4}},\]
and from \cref{trans inv}, $f_{i}(\theta)\ll_{g} g^{K}K^{-5/4}$ as well.
\end{proof}
It remains then to prove \cref{minor arc bound} for $\theta$ such that $\V{(g-1)\theta}\geq g^{-K+2\ell^{2}}$ and $w_{K}((g-1)\theta)\leq \ell$. We first recall the probabilistic model for digits set up in \cref{section prob model}. Let $X_{i} \st i=0,\ldots, K-1$ be i.i.d. copies of the random variable uniformly taking values in $\{0,\ldots, g-1\}-(g-1)/2$. Let 
\[\xb\coloneq \sum_{i=0}^{K-1}X_{i}g^{i},\]
and for $\mathcal{S}\subset \R$, let
\[\xb_{\mathcal{S}}\coloneq\sum_{i\in \mathcal{S}\cap \{0,\ldots, K-1\}}X_{i}g^{i}.\]
Recall from \cref{lemma expectation version f} that 
\[f_{i}(\theta)=g^{K}e(\theta(g^{K}-1)//2)\E_{\xb}e(\xb\theta)\ind_{\sx{\xb}=\xi_{i}}.\]
Here $\xi_{i}=\mu_{K}-k_{i}$, where $\mu_{K}\coloneq (g-1)K/2$, and $\sx{\xb}=\sum_{j=0}^{K-1}X_{j}$. Our aim in this section is to show the following proposition.
\begin{prop}\label{minor arc few non zero} Let $\xi\in \Supp(X)$ and
let $L,R$ be positive integers such that $\tfrac{3}{2} \log_{g} K\leq R\leq L$. Let $C>0$ be such that $\max(|\xi|,LR)\leq CK^{1/4}$. Suppose that $\theta\in\R$ is such that $w_{K}((g-1)\theta) \leq L$ and $\V{(g-1)\theta}\geq g^{-K+2LR}$. Then
 \[\E_{\xb}e(\xb\theta)\ind_{\sx{\xb}=\xi}\ll_{C,g} (|\xi|+LR)^{2}K^{-3/2}.\]
\end{prop}
This proposition allows us to now prove the minor arc bound, \cref{minor arc bound}, for all $\theta \in \mathfrak{m}$.
\begin{proof}[Proof of \cref{minor arc bound}]
\cref{FM very good g} and \cref{close to zero minor arc} give \cref{minor arc bound} for all $\theta\in \mathfrak{m}$ except those with $w_{K}((g-1)\theta)\leq \ell$ and $\V{(g-1)\theta}\geq g^{-K+2\ell^{2}}$. In this remaining case, \cref{minor arc few non zero} gives \cref{minor arc bound} upon taking $\xi=\xi_{i}$ and $R=L=\ell$; the choice of $\ell$ from \cref{l choice} ensures that the assumptions in the statement of the proposition hold. The constant $C$ in the statement of \cref{minor arc few non zero} can be taken to be some constant depending only on $g$ coming from \cref{l choice} and \cref{k cond close to av}.
\end{proof}
In our application of \cref{minor arc few non zero}, the parameters $L$ and $R$ are both taken to be $\ell\asymp_{g}\log K$. Despite this, it is convenient to separate the roles of $L$ and $R$ in the proof of \cref{minor arc few non zero}. The quantity $L$ controls the number of non-zero digits in the centred expansion of $(g-1)\theta$, and the quality of this approximation is controlled by $R$. By assuming that $LR\leq CK^{1/4}$, we are able to approximate $e(X\theta)$ by a small number $(\ll LR$ many) of the random variables $X_{i}$. 

More precisely, we will partition $\{0,\ldots, K-1\}$ into sets $\mathcal{D}$ and $\mathcal{E}$ which depend on the location of the non-zero digits in the centred base-$g$ expansion of $(g-1)\theta$. The digits indexed by $\mathcal{D}$ are those which have indices close to those of the non-zero digits of $(g-1)\theta$; these determine $e(\xb\theta)$ up to an error determined by $R$. The remaining digits, indexed by $\mathcal{E}$, vary randomly according to the constraint that $\sx{\xb}=\sx{\xbd}+\sx{\xbee}=\xi$. Crucially, the set $\mathcal{D}$ only indexes a small number of the digits of $X$, as $|\mathcal{D}|\ll LR$, so the assumption that $|\xi|+LR\ll_{C,g} K^{1/4}$ ensures that we can replace the condition $s(\xbee)=\xi-s(\xbd)$ by the simpler condition $s(\xbee)=0$. This allows us to decouple the averages over $\xbd$ and $\xbee$, roughly giving the following:
\begin{align}
    \E_{\xb}e(\xb\theta)\ind_{\sx{\xb}=\xi}&=\E_{\xbd}e(\xbd\theta)\E_{\xbee}e(\xbee\theta)\ind_{\sx{\xbd}+\sx{\xbee}=\xi}\nonumber\\
    &\approx \big ( \E_{\xbd}e(\xbd\theta)\big )\big(\E_{\xbee}e(\xbee\theta) \ind_{\sx{\xbee}=0}\big).\label{decouple outine}
\end{align}
This requires the local limit theorem, specifically the application in \cref{psi lemma}. We then show that either the average $\E_{\xbd}e(\xbd\theta)$ has sufficient cancellation to give \cref{minor arc few non zero}, or that $(g-1)\theta$ has an even more specific structure. In the latter case, we find that the average over $\xbee$ exhibits lots of cancellation, which requires the assumption that $\V{(g-1)\theta}\geq g^{-K+2LR}$.  

In order to have more control over the centred base-$g$ expansion of $\theta$, we prove \cref{minor arc few non zero} for 
$\theta\in \tfrac{1}{g-1}I_{g}$. From the definition of the interval $I_{g}$, given in \cref{Ig defn}, this ensures that the centred base-$g$ expansion of $(g-1)\theta$ has no integer part. Note this proves \cref{minor arc few non zero} in full generality: we can shift any $\theta \in I_{g}$ by an integer multiple of $1/(g-1)$ so that the translate lies in $\tfrac{1}{g-1}I_{g}$, and from \cref{trans inv}, shifting by a multiple of $1/(g-1)$ doesn't affect the absolute value of $|\E_{\xb}e(\xb\theta)\ind_{\sx{\xb}}|$. Let the centred base-$g$ expansion of $(g-1)\theta$ be the following, 
\begin{equation}\label{theta digit exp}
    (g-1)\theta=\sum_{j= 1}^{\infty}\varepsilon_{j}g^{-j}=\sum_{1\leq j \leq K}\varepsilon_{j}g^{-j}+\eta,
\end{equation}
where $\eta \coloneq \sum_{j>K}\varepsilon_{j}g^{-j}$, giving $|\eta|<g^{-K}$. Let $\{n_{j}\st1\leq j\leq w_{K}((g-1)\theta)+1\}$ index the first $w_{K}((g-1)\theta)+1$ non-zero digits of $(g-1)\theta$ after the radix point, so that \begin{equation}\label{n defn}
  1\leq n_{1}<n_{2}\ldots <n_{w_{K}((g-1)\theta)}\leq K<n_{w_{K}((g-1)\theta)+1},
\end{equation}
with $\varepsilon_{n_{j}}\neq0$ for $1\leq j\leq w_{K}((g-1)\theta)+1$ and $\varepsilon_{j}=0$ for all other $j$, $1\leq j\leq n_{w_{K}((g-1)\theta)+1}-1$.
Define $\mathcal{D}$ to be the set of indices which are within $R$ of the $n_{j}$, more precisely,
\begin{equation}\label{D set defn}
    \mathcal{D}=\bigcup_{j=1}^{w_{K}((g-1)\theta)+1}[n_{j}-R,n_{j}-1]\cap \{0,\ldots, K-1\}.
\end{equation}
These are the indices of $X_{i}$ which, up to a sufficiently small error, actually determine the value of $e(\xb\theta)$. From the assumption that $w_{K}((g-1)\theta)\leq L$, we have $|\mathcal{D}|\ll LR \leq CK^{1/4}$. For $0\leq r \leq g-2$, let 
\begin{equation}\label{E set defn}
    \mathcal{E}_{r}=\{i\notin \mathcal{D}\st \sum_{j\leq i}\varepsilon_{j}\equiv r \md{g-1}\}.
\end{equation}
The following lemma allows us to rigorously carry out the \enquote{decoupling} step sketched in \cref{decouple outine}.
\begin{lemma}\label{decoupling lemma} Let $\xi \in \Supp(X)$, and let $L,R$ be positive integers such that $\tfrac{3}{2} \log_{g} K\leq R\leq L$. Let $C>0$ be such that $\max(|\xi|,LR)\leq CK^{1/4}$. Suppose that $\theta\in I_{g}$ is such that $w_{K}((g-1)\theta) \leq L$ and suppose that $(g-1)\theta$ has centred base-$g$ expansion given by \cref{theta digit exp}. Let the sets $\mathcal{D}$ and $\mathcal{E}_{r}$ be as defined by \cref{D set defn} and \cref{E set defn} respectively for $0\leq r\leq g-2$. Let $x\in \{0,\ldots, g-2\}$ be such that $\mathcal{E}_{x}$ is maximal among the sets $\mathcal{E}_{r}$, and define the following $(g-1)$-tuple of integers,
\begin{equation}\label{a tuple defn}
    \mathbf{a}\coloneq(|\mathcal{E}_{x}|,|\mathcal{E}_{x+1}|,\ldots, |\mathcal{E}_{g-2}|,|\mathcal{E}_{0}|,|\mathcal{E}_{1}|,\ldots, |\mathcal{E}_{x-1}|).
\end{equation} 
Then
    \[\E_{\xb}e(\xb\theta)\ind_{\sx{\xb}=\xi}=\ru{x\xi}\E_{\xbd}e\Big(\xbd \Big(\theta-\frac{x}{g-1}\Big)\Big)\Psi(\mathbf{a};0)+O_{C,g}((|\xi|+LR)^{2}K^{-3/2}).\]
\end{lemma}
The definition of $\Psi(\mathbf{a};0)$ is given in \cref{psi defn}, and this term accounts for the average over $\xbee$ presented in the sketch \cref{decouple outine}. Recall that $\Psi(\mathbf{a};0)$ is a generalisation of the probability $P(T,0)$ to include certain $(g-1)$\sth roots of unity, where $P(T,0)$ is the probability that $T$ i.i.d. copies of the uniform random variable taking values in $\{0,\ldots, g-1\}-(g-1)/2$ sum to zero. In this instance, the value of $T$ is taken to be $K-|\mathcal{D}|$.

We prove \cref{decoupling lemma} in \cref{section decoupling}. The next lemma will be used to show that there is cancellation in the average over $\xbd$ if there is a non-zero digit in $(g-1)\theta$ followed by a string of zeros, and preceded by digits which have a sum congruent to $x$ modulo $g-1$, for the value $x$ defined by \cref{decoupling lemma}. 
\begin{lemma}\label{post cancellation} For $\theta\in \tfrac{1}{g-1}I_{g}$ suppose that $(g-1)\theta$ has centred base-$g$ expansion $\sum_{i=1}^{\infty}\varepsilon_{i}g^{-i}$. Suppose further that there is an index $m$, $1\leq m\leq K$, such that $\varepsilon_{m}\neq 0$, and $\varepsilon_{m+1}=\ldots= \varepsilon_{m+T}=0$, for some integer $T\geq 0$. Let $r\in \{0,\ldots, g-2\}$ be such that $\sum_{1\leq j \leq m}\varepsilon_{j}\equiv r\md{g-1}$. Then \[\E_{Y}e\Big(g^{m-1}Y\Big (\theta-\frac{r}{g-1}\Big)\Big)\ll g^{-T},\]
       where $Y$ is a random variable uniformly taking values in $\{0,\ldots, g-1\}-(g-1)/2$.
\end{lemma}
This will be proved in \cref{section d cancellation}. The next lemma will be used to show there is cancellation in the average over $\xbd$ when a rather different structure is present in the expansion of $(g-1)\theta$. In this case, we look for a non-zero digit preceded by a string of zeros, such that the digits preceding this have a sum that is not congruent to $x$ modulo $g-1$.  
\begin{lemma}\label{pre block cancellation}
For $\theta\in \tfrac{1}{g-1}I_{g}$ suppose that $(g-1)\theta$ has centred base-$g$ expansion $\sum_{i=1}^{\infty}\varepsilon_{i}g^{-i}$. Let $T$ be an integer such that $0\leq T\leq K$, and suppose that $\varepsilon_{m-T}=\ldots=\varepsilon_{m-1}=0$ for some index $m$, $T\leq m\leq K$. Let $r\in \{0,\ldots, g-2\}$ be such that $\sum_{1\leq j \leq m}\varepsilon_{j}\nequiv r\md{g-1}$. Then
\[\E_{Y_{m-T},\ldots, Y_{m-1}}e\Big (\sum_{j=m-T}^{m-1}Y_{j}g^{j}\Big(\theta-\frac{r}{g-1}\Big)\Big)\ll_{g} g^{-T},\]
where $Y_{m-T},\ldots, Y_{m-1}$ are i.i.d. uniform random variables taking values in $\{0,\ldots, g-1\}-(g-1)/2$.
\end{lemma}
This lemma will also be proved in \cref{section d cancellation}. Assuming these lemmas, we may now prove \cref{minor arc few non zero}.
\begin{proof}[Proof of \cref{minor arc few non zero}] As noted in the preceding sketch, we may assume that $\theta \in \tfrac{1}{g-1}I_{g}$. Therefore from \cref{decoupling lemma}, we have 
\begin{equation}\label{proof minor arc eq 1}
    \E_{\xb}e(\xb\theta)\ind_{\sx{\xb}=\xi}=\ru{x\xi}\E_{\xbd}e\Big(\xbd \Big(\theta-\frac{x}{g-1}\Big)\Big)\Psi(\mathbf{a};0)+O_{C,g}((|\xi|+LR)^{2}K^{-3/2}).
\end{equation}
for some $x\in\{0,\ldots, g-2\}$ and $\mathbf{a}$ as defined by \cref{a tuple defn}. To prove the proposition, it remains to show that the purported main term on the right hand side of \cref{proof minor arc eq 1} is also bounded by $O_{C,g}((|\xi|+LR)^{2}K^{-3/2})$. We will show that at least one of the following two inequalities always holds, so that either
\begin{equation}\label{main term cancellation D}
    \E_{\xbd}e\Big(\xbd \Big(\theta-\frac{x}{g-1}\Big)\Big)\ll_{g} g^{-R}
\end{equation}
or
\begin{equation}\label{main term cancellation psi}
    \Psi(\mathbf{a};0)\ll_{C,g} R^{2}K^{-3/2}.
\end{equation}
Note that as both $\Psi(\mathbf{a};0)$ and the average over $\xbd$ are trivially bounded by 1, either bound is sufficient to give \cref{minor arc few non zero}, using that $R\geq \tfrac{3}{2}\log_{g} K$ to bound $g^{-R}$. Let us first consider when \cref{main term cancellation D} holds. As the $X_{i}$ are independent, 
    \begin{equation}\label{av D product individual Xi}
        \E_{\xbd}e\Big(\xbd\Big (\theta-\frac{x}{g-1}\Big)\Big)=\prod_{j\in \mathcal{D}}\E_{X_{j}}e\Big(g^{j}X_{j}\Big (\theta-\frac{x}{g-1}\Big)\Big).
    \end{equation} 
If the centred expansion of $(g-1)\theta$ is such that the assumptions of \cref{post cancellation} are fulfilled for any integers $m,T$ with $1\leq m\leq K$ and $T\geq R$, and with the value of $r$ in the statement of \cref{post cancellation} equal to $x$, then
\[\E_{X_{m-1}}e\Big (X_{m-1}g^{m-1}\Big (\theta-\frac{x}{g-1}\Big )\Big )\ll g^{-R}.\]
Note that $X_{m-1}\in \mathcal{D}$ in this case, as $\varepsilon_{m}\neq 0$, so this average over $X_{m-1}$ appears in \cref{av D product individual Xi}. Therefore under these circumstances, we can use \cref{post cancellation} to show \cref{main term cancellation D}.

Similarly, we can use \cref{pre block cancellation} to show \cref{main term cancellation D}. Suppose that the centred base-$g$ expansion of $(g-1)\theta$ is such that there exists an integer $m$ with $R \leq m\leq K$ for which the assumptions in the statement of \cref{pre block cancellation} are fulfilled with $T=R$ and $r=x$. In this case, \cref{pre block cancellation} applies with $Y_{i}=X_{i}$ for $i=m-R,\ldots, m-1$ to give 
\[\E_{X_{[m-R,m-1]}}e\Big (X_{[m-R,m-1]}\Big (\theta-\frac{x}{g-1}\Big )\Big )\ll_{g} g^{-R}.\]
In order to use this to bound \cref{av D product individual Xi}, we require that the random variables $X_{m-R},\ldots, X_{m-1}$ are all contained in $\mathcal{D}$. Thus we also require that $\varepsilon_{m}\neq 0$ here.

The next claim shows that if \cref{main term cancellation D} doesn't hold, then the expansion of $(g-1)\theta$ has a very specific structure. Recall that the indexes labelled $n_{j}$ below are those defined by \cref{n defn}, which index the location of the non-zero digits within the first $K$ digits in the centred expansion of $(g-1)\theta$.
        \begin{claim}\label{no D canc} 
        If 
        \begin{equation}\label{D canc claim}
            \E_{\xbd}e\Big (\xbd\Big ( \theta-\frac{x}{g-1}\Big )\Big)\gg_{g} g^{-R}
        \end{equation}
       then $n_{j+1}-n_{j}\leq R-1 $ for $j=1,\ldots, w_{K}((g-1)\theta)-1$. Moreover, $\sum_{j=1}^{K}\varepsilon_{j}\nequiv x \md{g-1}$. 
    \end{claim}
    \begin{proof}[Proof of \cref{no D canc}]
We can assume there is at least one non-zero digit within the first $K$ digits of $(g-1)\theta$, as otherwise $\theta\in \mathfrak{M}$. Suppose that there are consecutive non-zero digits within the first $K$ digits of the centred expansion of $(g-1)\theta$ indexed by $v,w$ such that $w-v\geq R$. By assumption we have that $\varepsilon_{v+1}=\ldots= \varepsilon_{w-1}=0$. Let $a\coloneq \sum_{r\leq v}\varepsilon_{r}$. If $a\equiv x\md{g-1}$, then we can apply \cref{post cancellation} with $m=v$, $T=w-v\geq R$ and $r=x$ to obtain $\E_{X_{v-1}}e(X_{v-1}g^{v-1} ( \theta-x/(g-1) ))\ll g^{-R}$. As $\varepsilon_{v}\neq 0$, we have $v-1\in \mathcal{D}$, so this gives sufficient cancellation in the average over $\xbd$ to contradict \cref{D canc claim}. Therefore, $a\nequiv x\md{g-1}$. 
        
However, if $a\nequiv x\md{g-1}$ for $a$ as defined above, the assumptions of \cref{pre block cancellation} are now fulfilled with $m=w$, $T=R$ and $r=x$. This gives that $\E_{\xb_{[w-R,w-1]}}e(\xb_{[w-R,w-1]} ( \theta-x/(g-1) ))\ll_{g} g^{-R}$. Note that the interval $\{w-R,\ldots, w-1\}$ is contained in $\mathcal{D}$ by definition \cref{D set defn}, as $\varepsilon_{w}\neq 0$, so this again provides enough cancellation in the average over $\xbd$ to contradict \cref{D canc claim}.

Therefore we must have that each non-zero digit within the first $K$ digits of the centred expansion of $(g-1)\theta$ occurs within $R-1$ digits of another non-zero digit. Let $t\coloneq n_{w_{K}((g-1)\theta)}$ be the largest index of a non-zero digit occurring among the first $K$ digits in the centred expansion of $(g-1)\theta$.

To establish the second part of the claim, suppose that $\sum_{j=1}^{t}\varepsilon_{j}\equiv x \md{g-1}$. If $K-t\geq R$, then we can apply \cref{post cancellation} with $m=t$ and $T=R$, contradicting \cref{D canc claim}. On the other hand, if $K-t<R$, then $n_{1}>K-L(R-1)$ from the bound\begin{equation}\label{t-s bound}
           t-n_{1}=\sum_{j=1}^{w_{K}((g-1)\theta)-1}n_{j+1}-n_{j}\leq (L-1)(R-1).
       \end{equation}
       As $\V{(g-1)\theta}<g^{-n_{1}+1}$, this contradicts the assumption that $\V{(g-1)\theta}\geq g^{-K+2LR}$, so we must have $\sum_{j=1}^{t}\varepsilon_{j}\nequiv x \md{g-1}$.
        \end{proof}
\cref{no D canc} means that if there is insufficient cancellation in the average over $\xbd$ to show \cref{main term cancellation D}, then all the non-zero digits within the first $K$ digits of $(g-1)\theta$ occur very close together. In this case, we use this structure to show \cref{main term cancellation psi} holds, which proves \cref{minor arc few non zero}. To simplify notation, let \[ s\coloneq n_{1}, \qquad t\coloneq n_{w_{K}((g-1)\theta)} \qquad \mbox{and} \qquad   u\coloneq n_{w_{K}((g-1)\theta)+1}.\]

Suppose from now on that \cref{main term cancellation D} doesn't hold, and recall that $x$ is such that $|\mathcal{E}_{x}|$ is maximal among the $|\mathcal{E}_{r}|$. Let $a\coloneq \sum_{j=s}^{t}\varepsilon_{j}$; from \cref{no D canc}, $a\nequiv x \md{g-1}$. Let $y\in\{0,\ldots, g-2\}$ be such that $a\equiv y \md{g-1}$. Then the set $\mathcal{E}$ is partitioned into two intervals, $\mathcal{E}_{0}=\{0,\ldots, s-R\}$ and $\mathcal{E}_{y}=\{t,\ldots, \min(K-1,u-R-1)\}$, with $\mathcal{E}_{r}=\emptyset$ for $1\leq r\leq g-2$, $r\neq y$. Thus $x$ must either be $0$ or $y$.

First note that $y\neq 0$. If not, we would have $\mathcal{E}_{0}=\{0,\ldots, K-1\}\setminus \mathcal{D}$, and so $x=0$, as $|\mathcal{E}_{0}|\geq K-LR$ and $|\mathcal{E}_{r}|=0$ for $r\neq 0$. However this gives $a\equiv x \md{g-1}$, contradicting the second part of \cref{no D canc}.

Additionally, if $x=y$ then this contradicts the fact that $a\nequiv x \md{g-1}$. Therefore we must have $x=0$, whence $|\mathcal{E}_{0}|\geq |\mathcal{E}_{y}|$ by definition of $x$.

We use the assumption that $\V{(g-1)\theta}\geq g^{-K+2LR}$ to show that $|\mathcal{E}_{y}|\geq R$. As $|\mathcal{E}_{y}|=|\{t,\ldots , \min(K-1,u-R-1)\}|$, if $|\mathcal{E}_{y}|<R$ then $t>K-2R$. Recall from \cref{t-s bound} and the definition of $s$ and $t$ that $t-s\leq (L-1)(R-1)$. Combining these bounds, we see that $s>K-2R-(L-1)(R-1)$. As $\V{(g-1)\theta}\leq g^{-s+1}$, this contradicts the assumption that $\V{(g-1)\theta}\geq g^{-K+2LR}$.

Therefore $|\mathcal{E}_{0}|\geq |\mathcal{E}_{y}|\geq R$. Recall the definition of the tuple $\mathbf{a}$ from \cref{a tuple defn}; as $x=0$ we have that \[\mathbf{a}=(|\mathcal{E}_{0}|,0,\ldots,0,|\mathcal{E}_{y}|,0,\ldots,0 ).\] 
Our aim is to bound $\Psi(\mathbf{a};0)$ using \cref{psi 0 bound} with $m=R$ and $t=y$. To satisfy the assumptions of \cref{psi 0 bound} we need $R\ll_{C} |\mathcal{E}_{0}|^{1/4}$, which follows from the fact that $|\mathcal{E}_{0}|\geq (K-|\mathcal{D}|)/2\geq (K-LR)/2$. As $LR\leq CK^{1/4}$ by assumption, we have that $|\mathcal{E}_{0}|\gg_{C} K$, and certainly $R\leq CK^{1/4}$, giving $R\ll_{C} |\mathcal{E}_{0}|^{1/4} $. Applying \cref{psi 0 bound} gives the bound 
\begin{equation*}\label{psi cancellation}
    \Psi(\mathbf{a};0)\ll_{C,g} g^{-R}+R^{2}K^{-3/2}.
\end{equation*}
This gives \cref{main term cancellation psi}, using that $R\geq \tfrac{3}{2}\log_{g}K$.
\end{proof}
\subsection{Decoupling the averages over $\xbd$ and $\xbee$}\label{section decoupling}
In this section we prove the decoupling result, \cref{decoupling lemma}. This lemma allows us to replace the condition on the digits of $X$, $s(\xbd)+s(\xbee)=\xi$, with the condition $s(\xbee)=0$, even as the digits in $\xbd$ vary.
\begin{proof}[Proof of \cref{decoupling lemma}] Let $(g-1)\theta$ have centred base-$g$ expansion given by \cref{theta digit exp}. Dividing through by $(g-1)$ in this expansion gives
\[\theta=\frac{1}{g-1}\sum_{j=1}^{w}\varepsilon_{n_{j}}g^{-n_{j}}+\frac{\eta}{g-1},\]
where $|\eta|<g^{-K}$ and $w\coloneq w_{K}((g-1)\theta)$. Thus we can separate the contribution to $e(\xb\theta)$ from each non-zero digit as follows,
\begin{equation}\label{decoupling proof eq 1}
    e(\xb\theta)=e\Big(\xb\frac{\eta}{g-1}\Big)\prod_{j=1}^{w} e\Big(\xb \frac{\varepsilon_{n_{j}}}{(g-1)g^{n_{j}}}\Big).
\end{equation}
We have that
\[\Big|\xb_{[0,n_{j}-R-1]}\frac{\varepsilon_{n_{j}}}{(g-1)g^{n_{j}}}\Big|\leq g^{-R} \textrm{ and } \xb_{[n_{j},K-1]}\frac{\varepsilon_{n_{j}}}{g^{n_{j}}}\equiv X_{[n_{j},K-1]}\varepsilon_{n_{j}}\md{g-1}, \]
where the latter statement uses \cref{trans inv}. Thus the contribution to $e(X\theta)$ from the $n_{j}$\sth digit of $(g-1)\theta$ is
     \begin{align}
    e\Big(\xb\frac{\varepsilon_{j}}{(g-1)g^{n_{j}}}\Big)=e\Big(\xb_{[n_{j}-R,n_{j}-1]}\frac{\varepsilon_{n_{j}}}{(g-1)g^{n_{j}}}\Big )\ru{\varepsilon_{n_{j}}\xb_{[n_{j},K-1]}}+O(g^{-R})\label{ni contribution}.
    \end{align}
The $\eta$ term gives a similar contribution: let $u\coloneq n_{w_{K}((g-1)\theta)+1}-R$, then
    \begin{equation}\label{eta contribution}
        e(\xb\eta/(g-1))=e(\xb_{[u,K-1]}\eta/(g-1))+O(g^{-R}).
    \end{equation}
    If $u\geq K$, then the interval $[u,K-1]\cap \{0,\ldots, K-1\}$ is empty and $e(X\eta/(g-1))=1+O(g^{-R})$. From \cref{decoupling proof eq 1}, \cref{ni contribution} and \cref{eta contribution}, we have
\begin{align}
    \E_{\xb}e(\xb\theta)\ind_{\sx{\xb}=\xi}&=\E_{\xb}e\Big(\xb\frac{\eta}{g-1}\Big)\prod_{i=1}^{w}e\Big (\xb\frac{\varepsilon_{n_{i}}}{(g-1)g^{n_{i}}}\Big)\ind_{\sx{\xb}=\xi}\nonumber\\
    &=\E_{\xb}e\Big(\xb_{[u,K-1]}\frac{\eta}{g-1}\Big)\prod_{j=1}^{w}e\Big (\xb_{[n_{j}-R,n_{j}-1]}\frac{\varepsilon_{n_{j}}}{(g-1)g^{n_{j}}}\Big)\nonumber\\&\hspace{12em}\times \ru{\xb_{[n_{j},K-1]}\varepsilon_{n_{j}}}\ind_{\sx{\xb}=\xi}+O_{C,g}(L g^{-R}).\label{decoupling eg 2}\end{align}
    Here, we have used that $L\leq CK^{1/4}$ and $R\geq \tfrac{3}{2}\log_{g}K$ to obtain this error term.
Let $\mathcal{E}\coloneq \{0,\ldots, K-1\}\setminus \mathcal{D}$, and note that $\mathcal{E} $ is partitioned into $\mathcal{E}= \mathcal{E}_{0}\cup \ldots \cup \mathcal{E}_{g-2}$ for $\mathcal{E}_{r}$ defined by \cref{E set defn}. With this notation, we can rewrite \cref{decoupling eg 2} as
\begin{equation}\label{decoupling eq 2.5}\E_{X}e(X\theta)\ind_{s(X)=\xi}=\E_{\xbd}e(\xbd\theta)\E_{\xbee}\prod_{j=1}^{w}\ru{\xb_{[n_{j},K-1]\cap \mathcal{E}}\varepsilon_{j}}\ind_{\sx{\xb}=\xi}+O_{C,g}(L g^{-R}).\end{equation}
As the terms $n_{j}$ index precisely the non-zero digits of $(g-1)\theta$, each random variable $X_{u}$ for $u\in \mathcal{E}$ occurs in the above product once for each non-zero digit with index $n_{j}\leq u$, weighted by the value of that non-zero digit, $\varepsilon_{n_{j}}$. Rearranging the above product gives
        \[\prod_{j=1}^{w}\ru{\xb_{[n_{j},K-1]\cap \mathcal{E}}\varepsilon_{n_{j}}}=\prod_{u\in \mathcal{E}}\ru{X_{u}g^{u}\sum_{r\leq u}\varepsilon_{r}}=\prod_{v=1}^{g-2}\ru{v\xb_{\mathcal{E}_{v}}},\]
where the final equality follows from the definition of the sets $\mathcal{E}_{v}$ given by \cref{E set defn}. Hence from \cref{decoupling eq 2.5} and the above equation, $\E_{X}e(X\theta)\ind_{s(X)=\xi}$ equals
    \begin{equation}\label{decoupling eq 3}
        \E_{\xbd}e(\xbd\theta) \E_{\xbee}\ru{\xb_{\mathcal{E}_{1}}+2\xb_{\mathcal{E}_{2}}+\ldots+(g-2)\xb_{\mathcal{E}_{g-2}}}\ind_{\sx{\xbee}=\xi-\sx{\xbd}}+O_{C,g}(L g^{-R}).
    \end{equation}
Let $\mathbf{a}_{0}\coloneq(|\mathcal{E}_{0}|,\ldots,|\mathcal{E}_{g-2}|)$. From \cref{Psi expectation form} the average over $\xbee$ equals
\begin{align*}\label{decoupling psi for a0}
   \E_{\xbee}\ru{\xb_{\mathcal{E}_{1}}+2\xb_{\mathcal{E}_{2}}+\ldots+(g-2)\xb_{\mathcal{E}_{g-2}}}\ind_{\sx{\xbee}=\xi-\sx{\xbd}}=\Psi(\mathbf{a}_{0};\xi-\sx{\xbd}),
\end{align*}
and thus \cref{decoupling eq 3} equals
\[\E_{\xbd}e(\xbd\theta) \Psi(\mathbf{a}_{0};\xi-\sx{\xbd})+O_{C,g}(L g^{-R}).\]
Our aim is to remove the dependence on $\xbd$ from the term $ \Psi(\mathbf{a}_{0};\xi-\sx{\xbd})$ by applying \cref{psi lemma}. However, in order to apply \cref{psi lemma} with the tuple $\mathbf{a}_{0}$, we require that the first coordinate of $\mathbf{a}_{0}$, $|\mathcal{E}_{0}|$, satisfies $|\mathcal{E}_{0}|\geq |\mathcal{E}_{j}|$ for all $1\leq j \leq g-2$, which may not be the case. To circumvent this issue, we exploit the fact that multiplying \cref{decoupling eq 3} through by factors of $\ru{X}$ increases the coefficients of each $X_{\mathcal{E}_{r}}$, essentially allowing us to cycle the coordinates of $\mathbf{a}_{0}$. To do this, we use the following relation. As $\xb\equiv \sx{\xb} \md{g-1}$, we have $\E_{X}\ru{\xb}\ind_{\sx{\xb}=\xi}=\ru{\xi}$, and thus for any integer $m$,
 \begin{equation}\label{decoupling eq 5}
     \E_{\xb}e(\xb\theta)\ind_{\sx{\xb}}=\ru{m\xi}\E_{\xb}e(\xb\theta)\ru{-m\xb }\ind_{\sx{\xb}=\xi}.
 \end{equation}
Let $x$ be defined as in the statement of \cref{decoupling lemma}, that is, $x\in\{0,\ldots, g-2\}$ is such that $\mathcal{E}_{x}$ is maximal among the $|\mathcal{E}_{r}|$, and let $\mathbf{a}$ be as defined by \cref{a tuple defn},
\[  \mathbf{a}\coloneq(|\mathcal{E}_{x}|,|\mathcal{E}_{x+1}|,\ldots, |\mathcal{E}_{g-2}|,|\mathcal{E}_{0}|,|\mathcal{E}_{1}|,\ldots, |\mathcal{E}_{x-1}|).\]
From \cref{decoupling eq 5}, multiplying \cref{decoupling eq 3} through by $\ru{-x\xbee}$ gives
\begin{align}
    \E_{\xb}e(X\theta)\ind_{s(X)=\xi}&=\ru{x\xi}\E_{X}e(X\theta)\ru{-xX}\ind_{s(X)=\xi}\nonumber\\
    &=\ru{x\xi}\E_{\xbd}e\Big (\xbd\Big (\theta-\frac{x}{g-1}\Big )\Big )\E_{\xbee}\prod_{j=0}^{g-2}\ru{(j-x)\xb_{\mathcal{E}_{j}}}\ind_{\sx{\xbee}=\xi-\sx{\xbd}}\nonumber\\
    &=\ru{x\xi}\E_{\xbd}e\Big (\xbd\Big (\theta-\frac{x}{g-1}\Big )\Big )\Psi(\mathbf{a};\xi-\sx{\xbd}),\label{decoupling eq 4}
\end{align}
using \cref{Psi expectation form} and the definition of $\mathbf{a}$ for the final equality. As $|\mathcal{E}_{x}|$ is maximal among the sets $\mathcal{E}_{r}$, the tuple $\mathbf{a}$ fulfils the requirement that its first coordinate is the largest.
 
It remains to check that the other assumption for \cref{psi lemma} holds for the tuple $\mathbf{a}$ and $\nu\coloneq \xi-s(\xbd)$, namely that $|\nu|\leq C'|\mathcal{E}_{x}|^{1/4}$ for some $C'>0$. We have that $|\nu|\leq |\xi|+g|\mathcal{D}|/2\ll_{C,g} CK^{1/4}$, using the bound $|\mathcal{D}|\ll LR$ and the assumption $\max(|\xi|,LR)\leq CK^{1/4}$. Moreover, $|\mathcal{E}_{x}|\geq (K-|\mathcal{D}|)/(g-1)$ by definition of $x$, and thus $|\mathcal{E}_{x}|\gg_{C,g} K$. Therefore $|\nu|\ll_{C,g} |\mathcal{E}_{x}|^{1/4}$, and \cref{psi lemma} applies to give the bound
\begin{equation}\label{decoupling eq 6}
    \Psi(\mathbf{a};\xi-\sx{\xbd})=\Psi(\mathbf{a};0)+O_{C,g}((|\xi|+LR)^{2}K^{-3/2}).
\end{equation}
Combining \cref{decoupling eq 5}, \cref{decoupling eq 4} and \cref{decoupling eq 6} gives
    \[\E_{\xb}e(\xb\theta)\ind_{\sx{\xb}=\xi}=\ru{x\xi}\E_{\xbd}e(\xbd\theta)\ru{-\xbd x}\Psi(\mathbf{a};0)+O_{C,g}(Lg^{-R}+(|\xi|+LR)^{2}K^{-3/2}).\]
As $R\geq \tfrac{3}{2}\log_{g}K$, the $Lg^{-R} $ can be absorbed into the $(|\xi|+LR)^{2}K^{-3/2}$ term in the error term, giving the bound in the statement of the lemma.
\end{proof}
\subsection{Showing cancellation in the average over $\xbd$}\label{section d cancellation}
In this section, we prove \cref{post cancellation} and \cref{pre block cancellation}, which are used in the proof of \cref{minor arc few non zero} to give criteria on the structure of the centred base-$g$ expansion of $(g-1)\theta$ to get cancellation in $\E_{\xbd}e(\xbd(\theta-x/(g-1))).$ Let us restate \cref{post cancellation}.
\begin{post-canc-rpt}
  For $\theta\in \tfrac{1}{g-1}I_{g}$ suppose that $(g-1)\theta$ has centred base-$g$ expansion $\sum_{i=1}^{\infty}\varepsilon_{i}g^{-i}$. Suppose further that there is an index $m$, $1\leq m\leq K$, such that $\varepsilon_{m}\neq 0$, and $\varepsilon_{m+1}=\ldots= \varepsilon_{m+T}=0$, for some integer $T\geq 0$. Let $r\in \{0,\ldots, g-2\}$ be such that $\sum_{1\leq j \leq m}\varepsilon_{j}\equiv r\md{g-1}$. Then \[\E_{Y}e\Big(g^{m-1}Y\Big (\theta-\frac{r}{g-1}\Big)\Big)\ll g^{-T},\]
       where $Y$ is a random variable uniformly taking values in $\{0,\ldots, g-1\}-(g-1)/2$.
        \end{post-canc-rpt}
\begin{proof}
As $\varepsilon_{m+1}=\ldots=\varepsilon_{m+T}=0$ we have
    \begin{equation}\label{post canc theta structure}
        \theta=\frac{a}{(g-1)g^{m}}+t,
    \end{equation}
    where $a\coloneq \sum_{1\leq j\leq m}\varepsilon_{j}g^{m-j}$ and $t\coloneq\sum_{j\geq m+T+1}\varepsilon_{j}g^{-j}$. Note that $a\in \Z$ and $a\equiv r \md{g-1}$. We also have $|t|<g^{-m-T}$. Thus 
    \begin{align*}
        \theta-\frac{r}{g-1}&=\frac{a-g^{m}x}{(g-1)g^{m}}+t=\frac{a'}{g^{m}}+t,
    \end{align*}
    where $a'\in \Z$. Crucially, $a'\nequiv 0 \md{g}$ To see this, note that $a'\equiv a \equiv \varepsilon_{m}\md{g}$, and by assumption, $\varepsilon_{m}\neq 0$. Using this expression for $\theta$, we have
    \begin{equation}\label{post canc eq 1}
        e\Big(g^{m-1}Y\Big(\theta-\frac{r}{g-1}\Big)\Big)=e\Big(Y\frac{a'}{g}\Big)e(Ytg^{m-1})=e\Big(Y\frac{a'}{g}\Big)+O(g^{-T}).
    \end{equation}
   As $Y$ uniformly takes values in $\{0,\ldots, g-1\}$ and $a'\nequiv 0 \md{g}$, averaging \cref{post canc eq 1} over $Y$ gives the bound stated in the lemma, as $\E_{Y}e(Ya'/g)=0$.
\end{proof}
We now prove \cref{pre block cancellation}, which we also restate for convenience.
\begin{pre-canc-rpt}
For $\theta\in \tfrac{1}{g-1}I_{g}$ suppose that $(g-1)\theta$ has centred base-$g$ expansion $\sum_{i=1}^{\infty}\varepsilon_{i}g^{-i}$. Let $T$ be an integer such that $0\leq T\leq K$, and suppose that $\varepsilon_{m-T}=\ldots=\varepsilon_{m-1}=0$ for some index $m$, $T\leq m\leq K$. Let $r\in \{0,\ldots, g-2\}$ be such that $\sum_{1\leq j \leq m}\varepsilon_{j}\nequiv r\md{g-1}$. Then
\[\E_{Y_{m-T},\ldots, Y_{m-1}}e\Big (\sum_{j=m-T}^{m-1}Y_{j}g^{j}\Big(\theta-\frac{r}{g-1}\Big)\Big)\ll_{g} g^{-T},\]
where $Y_{m-T},\ldots, Y_{m-1}$ are i.i.d. uniform random variables taking values in $\{0,\ldots, g-1\}-(g-1)/2$.
\end{pre-canc-rpt}
\begin{proof}
    The assumptions on the digits of $(g-1)\theta$ ensure that $\theta$ has the following form:
    \[\theta=\frac{a}{(g-1)g^{m-T -1}}+t,\]
    where $a\nequiv r\md{g-1}$, and $|t|<g^{-m}$. To obtain this, take $a=\sum_{1\leq j\leq m-T -1}\varepsilon_{j}g^{j}$, and $t=\sum_{j\geq m}\varepsilon_{j}g^{-j}$. Then we have
    \begin{equation}\label{pre canc structure theta}
        \theta-\frac{r}{g-1}=\frac{a'}{(g-1)g^{m-T -1}}+t,
    \end{equation}
    where $a'=a-g^{m-T-1}r$. All we will require about $a'$ is that $a'\nequiv 0 \md{g-1}$, which follows from the fact that $a'\equiv a-r\md{g-1}$. 
    Let $Y\coloneq \sum_{j=m-T}^{m-1}Y_{j}g^{j}$; this uniformly takes values in $g^{m-T}\{0,\ldots, g^{T}-1\}-g^{m-T}(g^{T}-1)/2$. 
    Let
    \[\gamma \coloneq e\Big (-g^{m-T}\frac{g^{T}-1}{2}\Big (\theta-\frac{r}{g-1}\Big)\Big).\]
 As $Y$ is uniformly distributed, we have
    \begin{align}
        \E_{Y}e\Big (Y\Big (\theta-\frac{r}{g-1}\Big)\Big)&=\frac{\gamma}{g^{T}}\sum_{b=0}^{g^{T}-1}e\Big(g^{m-T}b\Big(\theta-\frac{r}{g-1}\Big )\Big)\nonumber\\
        &=\frac{\gamma}{g^{T}}\sum_{b=0}^{g^{T}-1}e\Big(g^{m-T}b\Big (\frac{a'}{(g-1)g^{m-T-1}}+t\Big )\Big)\nonumber\\
        &=\frac{\gamma}{g^{T}}\sum_{b=0}^{g^{T}-1}e\Big(\frac{ga'b}{g-1}+tg^{m-T}b\Big)\nonumber\\
        &=\frac{\gamma}{g^{T}}\frac{e\big (\frac{g^{T+1}a'}{g-1}+g^{m}t\big)-1}{e\big (\frac{ga'}{g-1}+g^{m-T}t\big)-1}.\label{pre canc eq 2}
    \end{align}
    As $ga'\nequiv 0 \md{g-1}$, we have that $|e(ga'/(g-1))-1|\geq |1-e(1/(g-1))|\eqqcolon \delta_{g}.$ Therefore, for $K$ sufficiently large,
    \[\Big|e\Big (\frac{ga'}{g-1}+g^{m-T}t\Big)-1\Big|=\Big|e\Big (\frac{ga'}{g-1}\Big)(1+O(g^{-T}))-1\Big|\geq \delta_{g}/2.\]
Thus from \cref{pre canc eq 2},
\[\Big |\E_{Y}e\Big(Y\Big (\theta-\frac{r}{g-1}\Big )\Big)\Big | \leq \frac{2}{g^{T}|e\big (\frac{ga'}{g-1}+g^{m-T}t\big)-1|}\leq \frac{4}{\delta_{g}g^{T}}.\]
\end{proof}

\section{Restricting to sums of Niven numbers}\label{section divisibility}
In this section we establish \cref{Main theorem N count} as a consequence of \cref{Main theorem S count}, by showing that the expected proportion of representations of $M$ as the sum of three integers with near-average digit sum are of sums of three Niven numbers. At the end of the section, we deduce \cref{Main theorem 1} from \cref{Main theorem N count} by giving an explicit choice of $k_{1},k_{2}$ and $k_{3}$ which satisfy the assumptions of \cref{Main theorem N count}.

Recall that $S_{i}=\{n<g^{K}\st s_{g}(n)=k_{i}\}$ for $i=1,2,3$ and integers $k_{1},k_{2}$ and $k_{3}$, where $s_{g}(n)$ denotes the base-$g$ digit sum of $n$. Throughout this section, we assume that $k_{1},k_{2}$ and $k_{3}$ satisfy the conditions given by \cref{k cond S} and \cref{k cond N}, which we restate here:
\begin{gather*}
    k_{1}+k_{2}+k_{3}\equiv M \md{g-1}, \hspace{1em} |k_{i}-\mu_{K}|\leq C_{g}, \nonumber\\
    (k_{i},g)=1 \textrm{ for }i=1,2,3 \textrm{ and }(k_{i},k_{j})=1 \textrm{ for }i,j\in \{1,2,3\} ,i\neq j.
\end{gather*}
The set of Niven numbers less than $g^{K}$ which have digit sum precisely $k_{i}$ is denoted by $\mathcal{N}_{i}$; these are the element of $S_{i}$ which are divisible by $k_{i}$. Let $g_{i}(\theta)$ be the Fourier transform $\widehat{\ind_{\mathcal{N}_{i}}}(\theta)$, so that 
\begin{equation*}
    g_{i}(\theta)=\sum_{n\in \mathcal{N}_{i}}e(n\theta).
\end{equation*}
The number of representations of $M$ as $n_{1}+n_{2}+n_{3}$ is
\begin{equation}\label{Eq counting n_1 n_2 n_3 reps}
    r_{\mathcal{N}_{1}+\mathcal{N}_{2}+\mathcal{N}_{3}}(M)=\int_{0}^{1}g_{1}(\theta)g_{2}(\theta)g_{3}(\theta)e(-M\theta)\dif \theta.
\end{equation}
To obtain an expression for $r_{\mathcal{N}_{1}+\mathcal{N}_{2}+\mathcal{N}_{3}}(M)$ in terms of $r_{S_{1}+S_{2}+S_{3}}(M)$, we use the following expression to link $g_{i}(\theta)$ and $f_{i}(\theta)$. By orthogonality, we can write $\ind_{k\mid n}$ as
\begin{equation*}
    \ind_{k\mid n}=\frac{1}{k}\sum_{j=0}^{k-1}e\big(\frac{jn}{k}\big).
\end{equation*}
From the definition of $\mathcal{N}_{i}$ and using the above equation we have
\begin{equation}\label{eq: g_i as sum over f translates}
    g_{i}(\theta)=\sum_{\substack{n\in S_{i}\\ k_{i}\mid n}}e(n\theta)=\frac{1}{k_{i}}\sum_{j=0}^{k_{i}-1}f_{i}\big(\theta+\frac{j}{k_{i}}\big).
\end{equation}
Here, $f_{i}(\theta)$ is the Fourier transform $\widehat{\ind_{S_{i}}}(\theta)$. From \cref{Eq counting n_1 n_2 n_3 reps} and \cref{eq: g_i as sum over f translates}, we obtain
\begin{equation}\label{count for niven numbers}
    r_{\mathcal{N}_{1}+\mathcal{N}_{2}+\mathcal{N}_{3}}(M)=\frac{1}{k_{1}k_{2}k_{3}}\sum_{\substack{j_{1},j_{2},j_{3}\\0\leq j_{i}<k_{i}}}\int_{0}^{1}\prod_{i=1}^{3}f_{i}\big(\theta+\frac{j_{i}}{k_{i}}\big)e(-M\theta)\dif \theta.
\end{equation}
Let $G(j_{1},j_{2},j_{3})$ denote the following,
\[G(j_{1},j_{2},j_{3})\coloneq \int_{0}^{1}\prod_{i=1}^{3}f_{i}\big(\theta+\frac{j_{i}}{k_{i}}\big)e(-M\theta)\dif \theta,\]
and let \begin{equation}\label{J defn}
    \mathcal{J}\coloneq\{(j_{1},j_{2},j_{3})\in \Z^{3} \st j_{i}\in \{0,\tfrac{k_{i}}{g-1},\ldots,\tfrac{(g-2)k_{i}}{g-1}\} \textrm{ for }i=1,2,3\}.
\end{equation} For all tuples $(j_{1},j_{2},j_{3})\in \mathcal{J}$, we will show that
\[G(j_{1},j_{2},j_{3})=r_{S_{k_{1}}+S_{k_{2}}+S_{k_{3}}}(M).\]
This is immediate for $G(0,0,0)$ from \cref{main S count}, and for other $(j_{1},j_{2},j_{3})\in \mathcal{J}$ (if they exist) we use the relation \cref{trans inv}. Suppose $j_{i}=a_{i}k_{i}/(g-1)$ for $i=1,2,3$ and integers $a_{i}$. Then from \cref{trans inv},
\[f_{i}\Big(\theta+\frac{a_{i}}{g-1}    \Big )=\ru{a_{i}k_{i}}f_{i}(\theta)=f_{i}(\theta),\]
as $(g-1)\mid a_{i}k_{i}$. Thus
\begin{equation}\label{to bound G contrib}
    r_{\mathcal{N}_{1}+\mathcal{N}_{2}+\mathcal{N}_{3}}(M)=\frac{|\mathcal{J}|}{k_{1}k_{2}k_{3}}r_{S_{1}+S_{2}+S_{3}}(M)+\frac{1}{k_{1}k_{2}k_{3}}\sum_{\substack{0\leq j_{i}<k_{i}\\(j_{1},j_{2},j_{3})\notin \mathcal{J}}}G(j_{1},j_{2},j_{3}).
\end{equation}
The next lemma is used to show that the remaining tuples $(j_{1},j_{2},j_{3})\notin \mathcal{J}$ contribute a negligible amount to \cref{to bound G contrib}. 
\begin{lemma}\label{G small}
Let $k_{1},k_{2},k_{3}$ be integers such that conditions \cref{k cond S} and \cref{k cond N} hold. Let $j_{1},j_{2},j_{3}$ be integers such that $(j_{1},j_{2},j_{3})\notin \mathcal{J}$ and $0\leq j_{i}\leq k_{i}-1$ for $i=1,2,3$. Then 
    \[G(j_{1},j_{2},j_{3})\ll_{g} g^{2K}K^{-29/6}.\]
\end{lemma}
To achieve this, we show that at any point $\theta\in [0,1]$, there exists $i\in \{1,2,3\}$ such that $f_{i}(\theta+j_{i}/k_{i})$ is very small, enough to kill off the contribution from the other terms. We are now ready to prove \cref{Main theorem N count} using \cref{Main theorem S count}, and assuming \cref{G small}. 
\begin{proof}[Proof of \cref{Main theorem N count}]
From \cref{to bound G contrib},
\[r_{\mathcal{N}_{1}+\mathcal{N}_{2}+\mathcal{N}_{3}}(M)=\frac{|\mathcal{J}|}{k_{1}k_{2}k_{3}}r_{S_{1}+S_{2}+S_{3}}(M)+\frac{1}{k_{1}k_{2}k_{3}}\sum_{\substack{0\leq j_{i}<k_{i}\\(j_{1},j_{2},j_{3})\notin \mathcal{J}}}G(j_{1},j_{2},j_{3}).\]
\cref{G small} gives  $G(j_{1},j_{2},j_{3})\ll_{g}g^{2K}K^{-29/6}$ for $(j_{1},j_{2},j_{3})\notin \mathcal{J}$. As there are at most $k_{1}k_{2}k_{3}$ such tuples, we have
\begin{equation}\label{N proof eq 1}
    r_{\mathcal{N}_{1}+\mathcal{N}_{2}+\mathcal{N}_{3}}(M)=\frac{|\mathcal{J}|}{k_{1}k_{2}k_{3}}r_{S_{1}+S_{2}+S_{3}}(M)+O_{g}(g^{2K}K^{-29/6}).
\end{equation}
This proves the first part of \cref{Main theorem N count}, as
\begin{equation}\label{J size}
    |\mathcal{J}|=\prod_{i=1}^{3}(g-1,k_{i})
\end{equation}
from the definition of $\mathcal{J}$ given in \cref{J defn}. To establish the second part of \cref{Main theorem N count}, we use \cref{Main theorem S count}. This gives that
\[r_{S_{1}+S_{2}+S_{3}}(M)=\frac{(g-1)M^{2}}{2(2\pi\sigma^{2}K)^{3/2}}+O_{g}(g^{2K}(\log K)^{4}K^{-7/4}),\]
and plugging this into \cref{N proof eq 1} gives
\[  r_{\mathcal{N}_{1}+\mathcal{N}_{2}+\mathcal{N}_{3}}(M)=\frac{(g-1)|\mathcal{J}|M^{2}}{2k_{1}k_{2}k_{3}(2\pi\sigma^{2}K)^{3/2}}+O_{g}\Big (\frac{|\mathcal{J}|g^{2K}(\log K)^{4}}{k_{1}k_{2}k_{3}K^{7/4}}+\frac{g^{2K}}{K^{29/6}}\Big).\]
As $|k_{i}-\mu_{K}|\ll_{g} 1$ from \cref{k cond S}, and $|\mathcal{J}|\leq (g-1)^{3}$, this error term is $O_{g}(g^{2K}(\log K)^{4}K^{-19/4})$. Additionally using that $\mu_{K}=(g-1)K/2$, we have
    \[\frac{1}{k_{1}k_{2}k_{3}}=\frac{8}{((g-1)K)^{3}}+O_{g}(K^{-4}).\]
Thus
\[r_{\mathcal{N}_{1}+\mathcal{N}_{2}+\mathcal{N}_{3}}(M)=\frac{4|\mathcal{J}|M^{2}}{(g-1)^{2}(2\pi\sigma^{2})^{3/2}K^{9/2}}+O_{g}\Big (\frac{g^{2K}(\log K)^{4}}{K^{19/4}}\Big).\]
Finally, using \cref{J size} we recover the main term stated in \cref{Main theorem N count}.
\end{proof}
\subsection{Translates of points with few non-zero digits}
In order to prove \cref{G small}, we isolate the following subset of $\R/\Z$, \begin{equation}\label{B defn}
    \fewnz\coloneq\{\theta \in \R/\Z \st w_{K}((g-1)\theta)\leq \ell\},
\end{equation} which includes the majors arcs, as well as some of the minor arc points. Here, $\ell\asymp_{g} \log K$ is given by \cref{l choice}. Recall that for $\theta\notin \fewnz$, \cref{FM very good g} gives the bound
\[f_{i}(\theta)\ll_{g} g^{K}K^{-12}.\]
The next proposition states that translates of $\theta$ by certain multiples of $k_{i}^{-1}$ cannot all simultaneously lie in $\mathcal{B}$. 
\begin{prop}\label{translate bad arcs overlap}
    Let $k_{1},k_{2},k_{3}$ be integers satisfying \cref{k cond S} and \cref{k cond N}, and let $j_{1},j_{2},j_{3}$ be integers such that $0\leq j_{i}\leq k_{i}-1$ and $(j_{1},j_{2},j_{3})\notin \mathcal{J}$. Then for $\mathcal{B}$ as defined by \cref{B defn},
    \[\Big (\fewnz -\frac{j_{1}}{k_{1}}\Big )\cap \Big (\fewnz -\frac{j_{2}}{k_{2}}\Big )\cap \Big (\fewnz -\frac{j_{3}}{k_{3}}\Big )=\emptyset.\]
\end{prop}
Assuming \cref{translate bad arcs overlap}, we can now prove \cref{G small}.
\begin{proof}[Proof of \cref{G small}]
    Assume that $k_{1},k_{2}$ and $k_{3}$ fulfil \cref{k cond S} and \cref{k cond N}. By two applications of H\"{o}lder's inequality, 
  \[|G(j_{1},j_{2},j_{3})|\leq \sup_{\theta\in [0,1]}\Big (\prod_{i=1}^{3}\Big |f_{i}\Big (\theta+\frac{j_{i}}{k_{i}}\Big ) \Big |\Big )^{1/3}\Big (\prod_{i=1}^{3}\int_{0}^{1}|f_{i}(\theta)|^{2} \Big)^{1/3}.\]
By Parseval's identity, and using the size estimates for each set $|S_{i}|$ from \cref{S set bound}, we have
\[G(j_{1},j_{2},j_{3})\ll_{g} \frac{g^{K}}{\sqrt{K}} \sup_{\theta\in [0,1]}\Big (\prod_{i=1}^{3}\Big |f_{i}\Big (\theta+\frac{j_{i}}{k_{i}}\Big ) \Big |\Big )^{1/3}. \]
Hence it suffices to show that for $(j_{1},j_{2},j_{3})\notin \mathcal{J}$, 
\begin{equation}\label{sup prod bound}
    \sup_{\theta\in [0,1]}\prod_{i=1}^{3}\Big |f_{i}\Big (\theta+\frac{j_{i}}{k_{i}}\Big ) \Big |=O_{g}(g^{3K}K^{-13}).
\end{equation}
For any $\theta\in [0,1]$, \cref{translate bad arcs overlap} gives that at most two of translates $\theta+j_{1}/k_{1},\theta+j_{2}/k_{2},\theta+j_{3}/k_{3}$ can lie in $\fewnz$. Suppose, without loss of generality, that these potential two bad translates are $j_{1}/k_{1}$ and $j_{2}/k_{2}$, so that $\theta+j_{1}/k_{1},\theta+j_{2}/k_{2}$ can potentially lie in $\mathcal{B}$ and $\theta+j_{3}/k_{3}\notin \mathcal{B}$.

We use the the trivial bound of $|S_{i}|\asymp_{g} g^{K}K^{-1/2}$ from \cref{S set bound} to bound $|f_{i}(\theta+j_{i}/k_{i})|$ for $i=1,2$. As $\theta+j_{3}/k_{3}\notin \mathcal{B}$, we use \cref{FM very good g} to bound ${|f_{3}(\theta +j_{3}/k_{3})|\ll_{g} g^{K}K^{-12}}$; combining these bounds gives \cref{sup prod bound}.
\end{proof} 

Our strategy for proving \cref{translate bad arcs overlap} is as follows. We first show that certain rationals have many non-zero digits in their centred base-$g$ expansions. Then we show that if $\alpha$ and $\beta$ both have very few non-zero digits in their centred base-$g$ expansions, so must $\alpha-\beta$. The conditions on the $k_{i}$ given by \cref{k cond N} are required precisely to ensure that rationals of the form 
\[\frac{(g-1)(j_{s}k_{t}-j_{t}k_{s})}{k_{s}k_{t}}\]
have many non-zero digits in their centred expansions, provided that $j_{s}$ and $j_{t}$ are not both integer multiples of $k_{s}/(g-1)$ or $k_{t}/(g-1)$ respectively. This is driven by the following lemma.
\begin{lemma}\label{1/k reciprocal}
    Let $j,k\in \N$ be such that $(k,g)=1$ and $j\nequiv 0\md{k}$. Then
    \[w_{K}(j/k)> \frac{K}{\lceil \log_{g}k\rceil}-1.\]
 \end{lemma}
\begin{proof}
    First, note the following fact. Suppose that $a,m$ are integers with $m>0$ and $a<g^{m}$. Then for $j,k$ as in the statement of the lemma, we have
    \begin{equation}\label{diophantine fact}
        \Big \Vert \frac{j}{k}-\frac{a}{g^{m}}\Big \Vert_{\R/\Z}\geq \frac{1}{kg^{m}}. 
    \end{equation}
    From the assumption that $j\nequiv 0\md{k}$, $j/k$ must have at least one non-zero digit after the radix point in its centred expansion. Moreover, there are infinitely many non-zero digits in the expansion, otherwise $j/k$ would be equal to a rational with a denominator that is a power of $g$, contradicting the fact $(k,g)=1$. Let $n_{0}=0$, and let $n_{i}$ denote the index of the $i$\sth non-zero digit after the radix point in the centred expansion of $j/k$. The proof of the lemma will follow from the next claim.
    \begin{claim}\label{claim 1/k}
    For all $i\geq 0$, $n_{i}\leq i\lceil \log_{g}k \rceil.$
    \end{claim}
    Assume that \cref{claim 1/k} holds, and let $M$ be the number of non-zero digits before the $(K+1)$\sth digit, such that $n_{M}\leq K< n_{M+1}$. From \cref{claim 1/k}, $K<n_{M+1}\leq (M+1)\lceil\log_{g} k\rceil$, giving $M>K/\lceil \log_{g} k\rceil -1$. 

    To prove \cref{claim 1/k} we induct on $i$, with the $i=0$ case following immediately by definition of $n_{0}$. Suppose the claim holds for $i=0,\ldots, t-1$, but that $n_{t}>t\lceil \log_{g} k\rceil$. This gives that $n_{t}-n_{t-1}>\lceil \log_{g}K\rceil$. Let $a/g^{n_{t-1}}$ be the rational obtained by truncating the centred base-$g$ expansion of $j/k$ at the $n_{t-1}$\sth digit. By definition of the $n_{i}$, all the digits of $j/k$ strictly between the $n_{t-1}$\sth and the $n_{t}$\sth are all zero, thus
    \[\Big \Vert \frac{j}{k}-\frac{a}{g^{m}}\Big \Vert_{\R/\Z}\leq \frac{1}{g^{n_{t}}}<\frac{1}{kg^{n_{t-1}}},\]
    where we have used that $n_{t}>n_{t-1}+\lceil \log_{g}K\rceil$. The above approximation for $j/k$ contradicts \cref{diophantine fact}, proving \cref{claim 1/k}.
\end{proof} 
We now show that the function $w_{K}$ is roughly additive.
\begin{lemma}\label{wk additive short ver}
    For all $\alpha,\beta \in \R$,
    \[w_{K}(\alpha-\beta)\ll_{g} w_{K}(\alpha)+w_{K}(\beta).\]
\end{lemma}
\begin{proof}
    From the upper bound on $w_{K}$ given by \cref{lemma: frac part to few non-zero}, 
    \[w_{K}(\alpha-\beta)\asymp_{g} \sum_{i=1}^{K}\V{g^{i}(\alpha-\beta)}^{2}.\]
    Furthermore, we have that
    \begin{equation}
        \V{\alpha-\beta}^{2}\leq (\V{\alpha}+\V{\beta})^{2}\leq 3\V{\alpha}^{2}+3\V{\beta}^{2}.
    \end{equation}
    for all $\alpha,\beta \in \R$. Thus
    \[\sum_{i=1}^{K}\V{g^{i}(\alpha-\beta)}^{2}\leq 3\sum_{i=1}^{K}(\V{g^{i}\alpha}^{2}+\V{g^{i}\beta}^{2})\ll_{g}w_{K}(\alpha)+w_{K}(\beta),\]
    now using the lower bound on $w_{K}$ from \cref{lemma: frac part to few non-zero}. Note that the implicit constant in the above inequality can be determined from \cref{lemma: frac part to few non-zero}, and only depends on $g$.
\end{proof}
Recall from \cref{J defn} that  
\[ \mathcal{J}\coloneq\{(j_{1},j_{2},j_{3})\in \Z^{3}\st j_{i}\in \{0,\tfrac{k_{i}}{g-1},\ldots,\tfrac{(g-2)k_{i}}{g-1}\}\cap \Z \textrm{ for }i=1,2,3\},\]
and from \cref{B defn} that $\mathcal{B}\coloneq\{\theta \in \R/\Z\st w_{K}((g-1)\theta)\leq \ell\}$, where $\ell\asymp_{g}\log K$. We may now prove \cref{translate bad arcs overlap}. 
\begin{proof}[Proof of \cref{translate bad arcs overlap}.] 
     Assume that $k_{1},k_{2}$ and $k_{3}$ fulfil conditions \cref{k cond S} and \cref{k cond N}, but that there exist integers $j_{1},j_{2},j_{3}$ with $0\leq j_{i}\leq k_{i}-1$ for $i=1,2,3$, and $(j_{1},j_{2},j_{3})\notin \mathcal{J}$ with
     \begin{equation}\label{overlap bad set exists}
         \Big (\fewnz -\frac{j_{1}}{k_{1}}\Big )\cap \Big (\fewnz -\frac{j_{2}}{k_{2}}\Big )\cap \Big (\fewnz -\frac{j_{3}}{k_{3}}\Big )\neq \emptyset.
     \end{equation}
     As $(j_{1},j_{2},j_{3})\notin \mathcal{J}$, there exists at least one index $i$, $i\in\{1,2,3\}$, such that
     \begin{equation}\label{j assumption}j_{i}\notin \Big \{0,\frac{k_{i}}{g-1},\ldots, \frac{(g-2)k_{i}}{g-1}\Big\}\cap \Z.\end{equation}
     Without loss of generality, assume that this term is $j_{1}$. We will show, without any assumptions on $j_{2}$, that
     \begin{equation}\label{overlap bad set with good j1}
         \Big (\fewnz-\frac{j_{1}}{k_{1}}\Big )\cap \Big (\fewnz -\frac{j_{2}}{k_{2}}\Big )=\emptyset.
     \end{equation}
     Suppose that \cref{overlap bad set with good j1} doesn't hold, so that there exists some $\theta$ such that $\theta\in (\fewnz -j_{1}/k_{1})\cap  (\fewnz -j_{2}/k_{2} )$. Then the following two translates of $\theta$, $\theta+j_{1}/k_{1},\theta+j_{2}/k_{2}$, both lie in $\mathcal{B}$ or in other words, both translates of $\theta$ simultaneously have few non-zero digits:
     \begin{equation}\label{simult few non-zero}
         w_{K}\Big ((g-1)\Big (\theta+\frac{j_{1}}{k_{1}}\Big )\Big)\leq \ell \textrm{ and } w_{K}\Big ((g-1)\Big (\theta+\frac{j_{2}}{k_{2}}\Big )\Big)\leq \ell.
     \end{equation}
Let $\alpha_{1}\coloneq (g-1)(\theta+j_{1}/k_{1})$ and $\alpha_{2}\coloneq (g-1)(\theta+j_{2}/k_{2})$, and note that 
\begin{equation}\label{alpha fraction}\alpha_{1}=\alpha_{2}+(g-1)\frac{j_{1}k_{2}-j_{2}k_{1}}{k_{1}k_{2}}.\end{equation}
From \cref{wk additive short ver}, and \cref{simult few non-zero}, we have
\begin{equation}\label{rational few nonzero}
    w_{K}\Big ((g-1)\frac{j_{1}k_{2}-j_{2}k_{1}}{k_{1}k_{2}}\Big )=w_{K}(\alpha_{1}-\alpha_{2})\ll_{g} w_{K}(\alpha_{1})+w_{K}(\alpha_{2})\ll_{g} \log K.
\end{equation}
Here we have used that $\ell \asymp_{g} \log K$ from \cref{l choice}. We will use this to derive a contradiction by showing that $(g-1)(j_{1}k_{2}-j_{2}k_{1})/k_{1}k_{2}$ has many non-zero digits in its centred expansion from \cref{1/k reciprocal}.

Let $j\coloneq (g-1)(j_{1}k_{2}-j_{2}k_{1})$, and $k\coloneq k_{1}k_{2}$. To see that the assumptions of \cref{1/k reciprocal} hold for the rational $j/k$, first note that we have $(k,g)=1$ from \cref{k cond N}. Furthermore, $j\nequiv 0 \md{k}$, as otherwise we would have $(g-1)j_{1}k_{2}\equiv 0\md{k_{1}}$. As $(k_{1},k_{2})=1$, this would imply that $(g-1)j_{1}\equiv 0\md{k_{1}}$, contradicting \cref{j assumption}. Therefore from \cref{wk additive short ver},
\begin{equation*}
    w_{K}\Big (\frac{(g-1)(j_{1}k_{2}-j_{2}k_{1})}{k_{1}k_{2}}\Big )> \frac{K}{\lceil \log_{g}k_{1}k_{2}\rceil}-1\gg_{g}\frac{K}{\log K},
\end{equation*}
where we have also used that $k_{1}k_{1}\asymp_{g} K^{2}$ from \cref{k cond S}. This contradicts \cref{rational few nonzero} for sufficiently large $K$.
\end{proof} 

\subsection{An explicit choice of $k_{1}, k_{2}$ and $k_{3}$}\label{section k choice}
We may now prove \cref{Main theorem 1}. This follows immediately from \cref{Main theorem N count}, provided that an appropriate choice of $k_{1},k_{2}$ and $k_{3}$ exist for all sufficiently large $K$ and $M$. Recall that $k_{1},k_{2},k_{3}$ are integers fulfilling the conditions given by \cref{k cond S} and \cref{k cond N}. We restate these conditions here:
\begin{gather*}
    k_{1}+k_{2}+k_{3}\equiv M \md{g-1}, \hspace{1em} |k_{i}-\mu_{K}|\leq C_{g}, \nonumber\\
    (k_{i},g)=1 \textrm{ for }i=1,2,3 \textrm{ and }(k_{i},k_{j})=1 \textrm{ for }i,j\in \{1,2,3\} ,i\neq j,
\end{gather*}
where \[C_{g}\coloneq g(g-1)\prod_{p\leq 10g^{2}}p.\] Note that the product in the above definition is over primes. We show that such a choice of $k_{1},k_{2}$ and $k_{3}$ exists for all $g,K$ and $M$.  Let $K_{0}$ be such that 
\begin{equation}\label{K0 choice }K_{0}\equiv 0\md{C_{g}}\textrm{ and }|K_{0}-\mu_{K}|\leq C_{g}/2.\end{equation}
We define $k_{i}\coloneq r_{i}+K_{0}$ for $i=1,2,3$, with $r_{i}$ to be determined by the following process. Let $r_{1}=1$, and let $r_{2}$ be the smallest prime greater than $g$. Clearly $|k_{1}-\mu_{K}|\leq C_{g}/2+1$, and we also have $|k_{2}-\mu_{K}|\leq C_{g}/2+2g+2$. Furthermore, as $k_{i}\equiv r_{i}\md{g}$ by definition of $C_{g}$, both $k_{1}$ and $k_{2}$ are coprime to $g$ for these choices of $r_{1}$ and $r_{2}$. 

It remains to choose $r_{3}$. As $k_{1}+k_{2}+k_{3}\equiv r_{1}+r_{2}+r_{3}\md{g-1}$, we need $r_{3}\equiv M-r_{2}-1\md{g-1}$. Let $a\in\{0,\ldots, g-2\}$ be such that $a\equiv M-r_{2}-1\md{g-1}$. We choose $\lambda\geq 0$ such that $r_{3}\coloneq a+\lambda(g-1)$ and so that the remaining conditions on $k_{3}$ are satisfied. 

If $a\neq 1$, let $\lambda_{0}\coloneq a-1$, otherwise let  $\lambda_{0}\coloneq g$. In either case, let $\lambda_{1}\coloneq \lambda_{0}+g$. We cannot have $r_{2}\mid a+\lambda_{1}(g-1)$ and $r_{2}\mid (a+\lambda_{0}(g-1))$, as this implies $r_{2}\mid g(g-1)$, which contradicts the fact that $r_{2}>g$ and $r_{2}$ is prime. Let $\lambda$ be either $\lambda_{0}$ or $\lambda_{1}$ as appropriate to ensure that $r_{2}\nmid r_{3}$.

By construction, $k_{1}+k_{2}+k_{3}\equiv M \md{g-1}$. As $-1\leq \lambda\leq 2g$, $|k_{3}-\mu_{K}|\leq C_{g}/2+a+|\lambda|(g-1)\leq C_{g}$. Moreover, $k_{3}\equiv r_{3}\equiv a-\lambda \md{g}$. In all cases, $\lambda\equiv a-1\md{g}$, and so $k_{3}\equiv 1\md{g}$, which ensures that $(k_{3},g)=1$. 

It remains to show that $k_{i},k_{j}$ are coprime. If not, any prime $q$ dividing both $k_{i}$ and $k_{j}$ must be bounded by $|k_{i}-k_{j}|$. However this is in turn bounded by $\max (|r_{1}-r_{2}|,|r_{2}-r_{3}|,|r_{1}-r_{3}|)<10g^{2}$, so we must have $q\mid C_{g}$. Hence $q\mid K_{0}$, and thus $q$ divides $r_{i}$ and $r_{j}$, contradicting the fact that $r_{1},r_{2},r_{3}$ are pairwise coprime by construction. 

\appendix
\section{Proof of \cref{FM bound 2}}\label{section FM bound}
The direct statement of the exponential sum bound, \cref{FM bound 2}, does not appear as a named theorem in \cite{fouvry2005entiers}. Instead, it appears in the proof of their result \citep[Theorem 1.2
]{fouvry2005entiers}. For completeness, we sketch the proof of \cref{FM bound 2} here, following the necessary parts of the proof of \citep[Theorem 1.2]{fouvry2005entiers}. Let us restate the result.
\begin{fm-rpt}For $\theta \in \R/\Z$,
     \[|f_{i}(\theta)| \leq g^{K}\exp \Big (-\frac{1}{2g}\sum_{i=0}^{K-1}\Vert g^{i}(g-1)\theta\Vertr^{2}\Big ).\]
\end{fm-rpt}
\begin{proof}
    By using orthogonality to detect the condition that $s_{g}(n)=k_{i}$ for $n\in S_{i}$, we may rewrite $f_{i}(\theta)$ as follows
    \[f_{i}(\theta)=\sum_{n\in S_{i}}e(n\theta)=\int_{0}^{1}e(-x k_{i})\sum_{n<g^{K}}e(n\theta)e(s(n)x)\dif x.\]
    To bound $|f_{i}(\theta)|$, it suffices to bound
    \begin{equation}\label{FM proof integral to bound}
        |f_{i}(\theta)|\leq \Big | \int_{0}^{1}\sum_{n<g^{K}}e(n\theta+s(n)x)\dif x\Big |.
    \end{equation}
    Note that from now on, the specific target digit sum does not appear in the proof; the only relevance of the exact value of $k_{i}$, in so far as this bound is concerned, is the strength of the bound relative to the trivial bound. To control the right hand side of \cref{FM proof integral to bound}, we use the recursive structure of the sum of digits function to isolate the contribution coming from each of the $K$ digits of $n<g^{K}$. We have
    \[\sum_{n<g^{K}}e(n\theta+s(n)x)=\prod_{\nu=0}^{K-1}\sum_{j=0}^{g-1}e(jg^{\nu}\theta+jx),\]
    and thus
    \begin{equation}\label{FM proof f bound}
        |f_{i}(\theta)|\leq \int_{0}^{1}\prod_{\nu=0}^{K-1}\Big |\sum_{j=0}^{g-1}e(jg^{\nu}\theta+jx)\Big |\dif x=\int_{0}^{1}\prod_{\nu=0}^{K-1}|U(g^{\nu}\theta+x)|\dif x
    \end{equation}
    where, following the notation of \cite{fouvry2005entiers}, $U(\alpha)\coloneq \sum_{j=0}^{g-1}e(j\alpha)$. The result \citep[Lemma 3.3]{fouvry2005entiers} gives that for any real $t,t_{0}$,
    \begin{equation}\label{FM lemma 3.3}
        |U(t)U(t+t_{0})|\leq g^{2}\exp\big (-\frac{1}{g}\Vert t_{0}\Vert^{2} \big ).
    \end{equation}
    We will apply \cref{FM lemma 3.3} with $t=g^{\nu}\theta+x$, and $t_{0}=(g-1)g^{\nu}\theta$ for $\nu=0,\ldots, K-2$, so that $t+t_{0}=g^{\nu+1}\theta+x$. Crucially, applying \cref{FM lemma 3.3} removes the dependency on the variable $x$. We have that
    \begin{align}
     \prod_{\nu=0}^{K-1}|U(g^{\nu}\theta+x)|^{2}&=|U(\theta+x)U(g^{K-1}\theta+x)|\prod_{\nu=0}^{K-2}|U(g^{\nu}\theta+x)U(g^{\nu+1}\theta+x)|\nonumber\\
     &\leq g^{2(K-2)}|U(\theta+x)U(g^{K-1}\theta+x)|\exp \Big (-\frac{1}{g}\sum_{\nu=0}^{K-2}\V{ (g-1) g^{\nu}\theta}^{2}\Big ).\label{squared ver}
    \end{align}
    We obtain
    \begin{align}
        \prod_{\nu=0}^{K-1}|U(g^{\nu}\theta+x)|& \leq g^{K-2}|U(\theta+x)U(g^{K-1}\theta+x)|^{1/2}\exp \Big (-\frac{1}{2g}\sum_{\nu=0}^{K-2}\V{ (g-1)g^{\nu}\theta}^{2} \Big ) \nonumber \\
        &\leq g^{K}\exp\Big (-\frac{1}{2g}\sum_{\nu=0}^{K-1}\V{ (g-1) g^{\nu}\theta}^{2}\Big).\label{U prod bound}
    \end{align}  
    The first inequality follows from taking the square root of the expression obtained in \cref{squared ver}. The final equality comes from the fact that $|U(\alpha)|\leq g$ for any $\alpha$, and that $g\exp(-\V{g^{K-1}(g-1)\theta}^{2}/2g)\geq ge^{-1/8g}\geq 1$ for $g\geq 2$. As this expression no longer depends on $x$, combining \cref{FM proof f bound} and \cref{U prod bound} gives the bound stated in the theorem. 
    \end{proof}
\section{Local limit theorem}\label{Section: LLT}
In \cref{section Psi} and \cref{Section: minor arc} we use the following local limit theorem statement, \cref{Local limit theorem application version}. Recall that $X_{0},\ldots, X_{K-1}$ are i.i.d. copies of $Y$, which uniformly takes values in $\{0,\ldots, g-1\}-(g-1)/2$. 
\begin{llt-rpt}[First part]
  Let $\nu$ be an integer. Then 
    \[\prob(X_{0}+\ldots+X_{K}=\nu)=\frac{e^{-x^{2}/2}}{\sqrt{2\pi\sigma^{2}K}}+O_{g}(K^{-3/2}),\]
    where $x=\nu/\sqrt{\sigma^{2}K}$, and $\sigma^{2}=(g^{2}-1)/12$ is the variance of $Y$. 
\end{llt-rpt}
This can be deduced immediately from far more general local limit theorems, such as Theorem 13 of \citep[Ch.~VII]{petrov1972independent}. The theorems in \citep[Ch.~VII]{petrov1972independent} allow one to obtain explicit expressions for arbitrarily many lower order terms and allow for much more general i.i.d. random variables. In order to show \cref{Local limit theorem application version} from one such more general theorem, it is necessary to calculate the cumulants of the random variable $Y$ to obtain the error term we require. As we are only concerned with a very specific family of uniform distributions, and don't require lower order terms, we can prove \cref{Local limit theorem application version} directly in this special case. We do so here, following the proofs of the more general theorems given in \cite{petrov1972independent}. 

We start by analysing the characteristic function of $Y$. Let $\cf(z)$ denote the characteristic function of $Y$, 
           \begin{align}
     \cf(z)&=\E_{Y} e^{izY}\nonumber\\
&=  1+\frac{2}{g}\sum_{j=1}^{(g-1)/2}\cos(jz)\textrm{ for odd }g \textrm{, and }\frac{2}{g}\sum_{j=0}^{(g-2)/2}\cos\Big (\frac{2j+1}{2}z\Big)\textrm{ for even }g.\label{cf Y}
\end{align} 
Around $z=0$, $\cf(z)$ has the following expansion.
\begin{lemma}\label{cf near zero}
        There exists $\varepsilon_{g}>0$ such that for $|z|\leq \varepsilon_{g}$, we have $\cf(z)\geq 1/2$ and 
        \begin{equation*}
            \cf(z)=1-\frac{\sigma^{2}z^{2}}{2}+O_{g}(z^{4}).
        \end{equation*}
    \end{lemma}
    \begin{proof}
        Let $\varepsilon_{g}>0$ be sufficiently small such that 
        \[\cos (jz)\geq 1/2 \textrm{ for }j\in [1,g/2]. \]
        For odd $g$, using \cref{cf Y} we obtain:
        \begin{align*}
            \cf(Y)=\frac{1}{g}\Big (1+2\sum_{j=1}\cos(jz) \Big)\geq \frac{1}{g}\Big (1+\frac{g-1}{2}\Big )\geq \frac{1}{2},
        \end{align*}
       and for even $g$,
\[\cf(z)=\frac{2}{g}\sum_{j=0}^{(g-2)/2}\cos\Big (\frac{2j+1}{2}z\Big)\geq \frac{1}{2},\]
thus proving the first part of the claim. Moreover, $\varepsilon_{g}$ only depends on $g$, and so uniformly in $z$, depending only on $g$, we can expand $\cos(jz)$ in this range 
        \[\cos(jz)=1-\frac{j^{2}z^{2}}{2}+O_{g}(z^{4}), \textrm{ for }j\in [1,g/2].\]
         Plugging this expansion into \cref{cf Y} and using that $\sigma^{2}=(g^{2}-1)/12$, we see that for odd $g$,
         \begin{align*}
             \cf(z)=\frac{1}{g}+\frac{2}{g}\sum_{j=1}^{(g-1)/2}\cos(jz)&=1-\frac{z^{2}}{g}\sum_{j=1}^{(g-1)/2}j^{2}+O_{g}(z^{4})=1-\frac{\sigma^{2}z^{2}}{2}+O_{g}(z^{4}),
         \end{align*}
         and for even $g$,
         \begin{align*}
             \cf(z)=\frac{2}{g}\sum_{j=0}^{(g-2)/2}\cos\Big (\frac{2j+1}{2}z\Big)&=1-\frac{z^{2}}{g}\sum_{j=0}^{(g-2)/2}\big(\frac{2j+1}{2}\Big)^{2}+O_{g}(z^{4})\\
             &=1-\frac{\sigma^{2}z^{2}}{2}+O_{g}(z^{4}).\qedhere
         \end{align*}
    \end{proof}
    We may now prove \cref{Local limit theorem application version}.
   \begin{proof}[Proof of \cref{Local limit theorem application version}]
    By orthogonality, 
      \begin{align}  \prob(X_{1}+\ldots+X_{K}=\nu)&=\E\Big(\frac{1}{2\pi}\int_{-\pi}^{\pi}e^{-it\nu}e^{it(X_{1}+\ldots+X_{K})}\dif t\Big)=\frac{1}{2\pi}\int_{-\pi}^{\pi}e^{-it\nu}\cf(t)^{K}\dif t\nonumber\\ 
        &=\frac{1}{2\pi\sigma K^{1/2}}\int_{-\pi\sigma K^{1/2}}^{\pi\sigma K^{1/2}}e^{-itx}\cf\Big(\frac{t}{\sigma K^{1/2}}\Big)^{K}\dif t,\label{cf integral}
    \end{align}
    where $x=\nu/\sigma K^{1/2}$. To establish \cref{Local limit theorem application version} then, it suffices to show that
    \begin{align}\label{int for llt to show}
        \int_{-\pi\sigma K^{1/2}}^{\pi\sigma K^{1/2}}e^{-itx}\cf\Big(\frac{t}{\sigma K^{1/2}}\Big)^{K}\dif t=\sqrt{2\pi}e^{-x^{2}/2}+O_{g}(K^{-1}).
    \end{align}    
    We approximate the characteristic function $\cf(t/\sigma K^{1/2})$ close to 0 to obtain the main term in \cref{Local limit theorem application version}, and bound the remaining contribution by $K^{-3/2}$. Let $\varepsilon_{g}$ be as in \cref{cf near zero}. We show that for $|t|< K^{1/4}\varepsilon_{g}$, 
    \begin{equation}\label{char function approx}
        \Big |\cf\Big(\frac{t}{\sigma K^{1/2}}\Big)^{K}-e^{-t^{2}/2} \Big |\ll_{g} e^{-t^{2}/2}t^{4}K^{-1}.
    \end{equation}
    To show \cref{char function approx}, we expand $\log \cf(z)$ in terms of $z$. For $|z|\leq \varepsilon_{g}$, we have $\varphi_{Y}(z)\geq 1/2$ and in particular, $\cf(z)$ is positive in this range so $\log \cf(z)$ is defined. For $|z|\leq \varepsilon_{g}$, let $w=1-\cf(z)$. We have that $|w|\leq 1/2$ and thus
    \begin{align*}
        \log \cf(z)&=\log (1-w)=-w+O_{g}(w^{2}).
    \end{align*}
From \cref{cf near zero} and the definition of $w$, we have
    \begin{equation}\label{log cf expansion}
        \log \cf(z)=-\frac{\sigma^{2}z^{2}}{2}+O_{g}(z^{4})
    \end{equation}
    for $|z|\leq \varepsilon_{g}$. We apply \cref{log cf expansion} with $z=t/\sigma K^{1/2}$ in the range $|t|\leq K^{1/4}\varepsilon_{g}$. Note that $\sigma^{-2}K^{-1/4}\leq 1$ for all $g\geq 2$ and $K\geq 1$, so in this range, $|t/\sigma K^{1/2}|\leq \varepsilon_{g}$.
       
    Hence we have
    \[\log \cf\Big(\frac{t}{\sigma K^{1/2}} \Big)^{K}=-\frac{t^{2}}{2}+O_{g}(t^{4}K^{-1})\]
    which gives
    \[\cf\Big(\frac{t}{\sigma K^{1/2}} \Big)^{K}=e^{-t^{2}/2}\exp \big( O_{g}(t^{-4}K^{-1})\big).\]
    Using that $e(O_{g}(t^{-4}K^{-1}))=1+O_{g}(t^{-4}K^{-1})$, as $|t|\leq K^{1/4}\varepsilon_{g}$, gives the bound stated in \cref{char function approx}. We can now estimate \cref{cf integral} by evaluating the following integrals:
    \begin{gather*}I_{1}\coloneq \int_{-\infty}^{\infty }e^{itx}e^{-t^{2}/2}\dif t, \hspace{0.5em}I_{2}\coloneq \int_{|t|>K^{1/4}\varepsilon_{g}}e^{itx}e^{-t^{2}/2}\dif t\\
   I_{3}\coloneq \int_{|t|<K^{1/4}\varepsilon_{g}}\Big | \cf\Big (\frac{t}{\sigma K^{1/2}}\Big )^{K}-e^{-t^{2}/2}\Big |\dif t, \hspace{0.5em} I_{4}=\int_{K^{1/4}\varepsilon_{g}\leq |t|\leq \sigma K^{1/2}}\Big | \cf\Big (\frac{t}{\sigma K^{1/2}}\Big )\Big |^{K}\dif t.\end{gather*}
Thus 
\[ \int_{-\pi\sigma K^{1/2}}^{\pi\sigma K^{1/2}}e^{-itx}\cf\Big(\frac{t}{\sigma K^{1/2}}\Big)^{K}\dif t=I_{1}+O_{g}(I_{2}+I_{3}+I_{4}).\]
The integral $I_{1}$ is 
    \[\int_{-\infty}^{\infty}e^{itx}e^{-t^{2}/2}\dif t=\sqrt{2\pi}e^{-x^{2}/2}.\]
Thus in order to show \cref{int for llt to show}, it remains to bound the other integral terms. We have 
    \[I_{2}\leq \int_{|t|>K^{1/4}\varepsilon_{g}}e^{-t^{2}/2}\dif t\ll_{g} K^{-10} \textrm{ and }I_{3}\ll \frac{1}{K}\int_{|t|<K^{1/4}\varepsilon_{g}}e^{-t^{2}/2}t^{4}\dif t\ll_{g} \frac{1}{K},\]
using \cref{char function approx} for $I_{3}$. Finally we turn to $I_4$. For $t\in [K^{1/4}\varepsilon_{g},\sigma K^{1/2}]$, we have that $|\cf(t/\sigma K^{1/2})|\leq \cf(\varepsilon_{g}\sigma K^{1/4})$. To see that $\cf(\varepsilon_{g}\sigma K^{1/4})>0$, we use \cref{cf near zero} to expand the characteristic function
 \[\cf(\varepsilon_{g}\sigma K^{1/4})=1-\frac{\varepsilon_{g}^{2}}{2K^{1/2}}+O_{g}(K^{-1})\leq 1-\frac{\varepsilon_{g}^{2}}{4K^{1/2}}\]
 for $K$ sufficiently large in terms of $g$. Thus
 \[I_{4}\leq \int_{|t|\in [K^{1/4}\varepsilon_{g},\sigma K^{1/2}]}\Big (1-\frac{\varepsilon_{g}}{4\sigma^{2}K^{1/2}}\Big )^{K}\leq \sigma K^{1/2}\Big(1-\frac{\varepsilon_{g}}{4\sigma^{2}K^{1/2}}\Big )^{K} \ll_{g}K^{-10}. \qedhere\]
 \end{proof}
\bibliographystyle{abbrv}
\bibliography{bibliography} 
\end{document}